\documentclass[a4paper]{amsart}

\usepackage{xspace}
\usepackage{graphicx}

\numberwithin{equation}{section}

\newcommand{\defas}{:=}
\newcommand{\eps}{\epsilon}
\newcommand{\grd}{\nabla}
\DeclareMathOperator{\spn}{span}

\newcommand{\N}{\mathbb{N}}
\newcommand{\R}{\mathbb{R}}
\newcommand{\Rd}{\R^d}
\newcommand{\opnint}[2]{{]}#1,#2{[}}
\newcommand{\clsint}[2]{[{#1},{#2}]}

\newcommand{\ocint}[2]{{]{#1}},{#2]}}
\newcommand{\clos}[1]{\overline{#1}}

\DeclareMathOperator{\Conv}{Conv}
\DeclareMathOperator{\diam}{diam}
\DeclareMathOperator{\Vertices}{Vert}

\newcommand{\vol}[1]{\left|#1\right|}
\newcommand{\area}[1]{\left|#1\right|}

\newcommand{\error}[2]{\left|#1\right|_{#2}}
\newcommand{\verror}[2]{\left\|#1\right\|_{#2}}
\newcommand{\norm}[2]{\left\|#1\right\|_{#2}}
\newcommand{\seminorm}[2]{\left|#1\right|_{#2}}
\newcommand{\Hil}[1]{H^{#1}}
\newcommand{\Leb}[1]{L^{#1}}
\newcommand{\Sob}[2]{W^{#1,#2}}
\newcommand{\Lebnorm}[3]{\norm{#1}{0,#2;#3}}

\newcommand{\Sobseminorm}[4]{\seminorm{#1}{#2,#3;#4}}

\newcommand{\Ccmpl}{{C_\text{cmpl}}}
\newcommand{\CTr}{{C_\text{Tr}}}
\newcommand{\Cnstr}{\mathcal{C}}
\newcommand{\DBFct}{\Psi}
\newcommand{\DcConst}{\delta}
\newcommand{\elm}{K}

\newcommand{\Face}{F}
\newcommand{\Faces}{\mathcal{F}_{d-1}}
\newcommand{\Ip}{\Pi}
\newcommand{\IpSZ}{\Pi^{\text{SZ}}}

\newcommand{\LgrNodes}[1]{L_{#1}}
\newcommand{\LgrBFct}{\Phi}
\newcommand{\NodalVar}{N}
\newcommand{\MIdxOrd}[1]{|{#1}|}
\newcommand{\mesh}{\mathcal{M}}
\newcommand{\Poly}[2]{\mathcal{P}_{#1}({#2})}
\newcommand{\pdeg}{\ell} 
\newcommand{\refmnts}{\mathbb{M}}
\newcommand{\Simplex}{\elm}
\newcommand{\Simplexd}{{\Simplex_d}}
\newcommand{\Skeleton}{\Sigma}
\newcommand{\Splines}{S}
\newcommand{\ve}{\underline{e}}

\DeclareMathOperator{\supp}{supp}

\newcommand{\apriori}{a~priori\xspace}

\theoremstyle{plain}
\newtheorem{thm}{Theorem}[section]
\newtheorem{cor}{Corollary}[section]
\newtheorem{prop}{Proposition}[section]
\newtheorem{lem}{Lemma}[section]
\theoremstyle{definition}
\newtheorem{rem}{Remark}[section]

\begin{document}
%
%
%

%
%
%
%

\title[Piecewise polynomial approximation of gradients]{Approximating 
gradients with continuous piecewise polynomial functions}
\date{\today}

\author[A. Veeser]{Andreas Veeser}
\email{Andreas.Veeser@unimi.it}
\urladdr{http://www.mat.unimi.it/users/veeser/}


\begin{abstract}
Motivated by conforming finite element methods for elliptic problems of second 
order, we analyze the approximation of the gradient of a target function by 
continuous piecewise polynomial functions over a simplicial mesh.  The main 
result is that the global best approximation error is equivalent to an 
appropriate sum in terms of the local best approximations errors on elements.  
Thus, requiring continuity does not downgrade local approximability and 
discontinuous piecewise polynomials essentially do not offer additional 
approximation power, even for a fixed mesh.  This result implies error bounds in 
terms of piecewise regularity over the whole admissible smoothness range.  
Moreover, it allows for simple local error functionals in adaptive tree 
approximation of gradients.
\end{abstract}

\maketitle

\section{Introduction}
\label{S:Intro}
%
%

Finite element methods are one of the most successful tools for the
numerical solution of partial differential equations.  In their
simplest form they are Galerkin methods where the discrete space is
given by elements that are appropriately coupled.  This piecewise
structure allows to construct bases that are, on one hand, relatively
easy to implement and, on the other hand, are locally supported.  The
latter leads to linear systems with sparse matrices, which can be
stored and often solved with optimal linear complexity.

\medskip This article concerns the approximation properties of continuous
piecewise polynomial functions over a simplicial mesh, which build a prototype 
finite element space.  It analyzes the interplay of global and local best 
errors when approximating the gradient of a target function.

Continuous piecewise polynomial functions arise when solving elliptic boundary 
value problems of second order with Lagrange elements (see \S\ref{S:Poly} for 
a definition).  To be more specific, if the associated bilinear form is 
$H^1_0(\Omega)$-coercive, a typical choice is
\begin{equation}
\label{model_subspace}
 \Splines
 \defas
 \Big\{ v:\clos\Omega\to\R \mid 
  \forall\elm\in\mesh\ v_{|\elm}\in\Poly{\pdeg}{\elm}, \,
  v\in C^0(\clos\Omega), \, v_{|\partial\Omega}=0
 \Big\},
\end{equation}
where $\mesh$ is a conforming simplicial mesh of a domain $\Omega$ and 
$\Poly{\pdeg}{\elm}$ denotes the set of polynominals with degree $\leq\pdeg$ 
over an element $\elm\in\mesh$.  Requiring continuity and incorporating the 
boundary condition in \eqref{model_subspace} ensure $\Splines \subset 
H^1_0(\Omega)$ and thus the conformity of finite element method.
C\'ea's lemma therefore implies that the error of the Galerkin solution in
$\Splines$ is dictated by the best global approximation error for the
exact solution $u\in H^1_0(\Omega)$:
\begin{equation*}
\label{best-global-error}
 E(u)
 \defas
 E(u,\Splines)
 \defas
 \inf \big\{
  \norm{\grd(u-v)}{\Omega} \mid v\in\Splines
 \big\},
\end{equation*}
where $\norm{\cdot}{\Omega}$ stands for the norm of $L^2(\Omega)$.  In
view of the piecewise structure of space $\Splines$, the approximation of $u$ 
on each element $\elm\in\mesh$ is limited by the local shape functions 
$\Poly{\pdeg}{\elm}$.  This suggests to introduce the local best approximation 
errors
\begin{equation*}
\label{local-best-errors}
 e_\elm(u)
 \defas
 \inf \big\{
  \norm{\grd(u-p)}{\elm} \mid p\in\Poly{\pdeg}{\elm}
 \big\},
\qquad
 \forall\elm\in\mesh.
\end{equation*}
The question arises how the global and local best errors are related and how 
their relationship is affected by requiring conformity.

In the described context the main result Theorem \ref{T:elm-dec} reads as 
follows.  For any conforming mesh, the global best error is equivalent to the 
appropriately collected local best errors.  More precisely, there holds
\begin{equation}
\label{equivalence}
 \left(
  \sum_{\elm\in\mesh} e_\elm(u)^2
 \right)^{\frac12}
 \leq
 E(u)
 \leq
 C\left(
  \sum_{\elm\in\mesh} e_\elm(u)^2
 \right)^{\frac12},
\end{equation}
where $C$ can be bounded in terms of the shape regularity of $\mesh$.  The 
first inequality in \eqref{equivalence} is straight-forward and just a 
quantitative version of the motivation that suggests to introduce the local 
best errors.  The second inequality, which is not obvious and the proper 
concern of this paper, means that the above requirements for conformity 
essentially do not downgrade the approximation potiential given by the local 
discrete spaces.  It thus confirms in particular the coupling via continuity of 
the elements.
Adopting the broken $H^1$-seminorm as error notion, there is a second 
interpretation of \eqref{equivalence}: discontinuous and continuous piecewise 
polynomial functions have essentially the same approximation power.  We also 
derive variants of \eqref{equivalence} addressing the coupling of partial 
derivatives of the approximants and mesh conformity; see Theorem 
\ref{T:pder-dec} and Theorem \ref{T:n-elm-dec}, respectively.

\smallskip The second inequality in \eqref{equivalence} is proved in 
\S\ref{S:ConformityImpact} by means of suitable local error bounds for a 
continuous interpolant.  As interpolant, one can use a variant of the 
Scott-Zhang interpolant \cite{Scott.Zhang:90} or averages like in S.~Brenner 
\cite{Brenner:96} of local best approximations.  The key ingredients for these 
local error bounds are the Trace and Poincar\'e inequalities.

\medskip Equivalence \eqref{equivalence} reduces the quantification of
the global best error to the quantification of decoupled, local best
approximation errors.  We illustrate the usefulness of this aspect
with two applications in \S\ref{S:appl}.

First, the second inequality in \eqref{equivalence} may be used in the
\apriori analysis of finite element solutions.  For example, inserting
the Bramble-Hilbert lemma in the right-hand side of \eqref{equivalence}, one 
readily obtains the upper bound
\begin{equation}
\label{apriori-ub}
 E(u)
 \leq
 C \left(
  \sum_{\elm\in\mesh} h_\elm^{2\ell}
   \norm{D^{\ell+1} u}{\elm}^2
 \right)^{\frac12},
\end{equation}
where $h_\elm$ denotes the diameter of an element $\elm\in\mesh$.  Notice that 
the right-hand side involves only piecewise regularity and so vanishes whenever
$u\in\Splines$.  This is also true for Lagrange interpolation error estimates, 
but not for the available error bounds \cite{Clement:75,Scott.Zhang:90} for 
interpolation of non-smooth functions.  Here it is obtained without invoking 
the imbedding $H^{\ell+1}\subset C^0$ and so also under weaker regularity 
assumptions on the target function $u$. 

Second, \eqref{equivalence} can be applied in constructive nonlinear 
approximation.  When using the local best errors $e_\elm(u)$ as local error 
functionals in the adaptive tree approximation of P.~ Binev and R.~DeVore 
\cite{Binev.DeVore:04}, then equivalence \eqref{equivalence} ensures that the 
approximations, which are constructed with linear complexity, are near best 
with respect to the $H^1$-seminorm. 

\section{Continuous piecewise polynomial functions and gradients}
\label{S:Poly}
%
%
In this section we define the approximants, fix associated notation, and review 
their relationship with gradients.  We also provide a basis for them and, to 
prepare interpolation, link the coefficients of that basis to the space of 
target functions.

\bigskip Let $\Omega$ be a non-empty open set of $\Rd$, $d\in\N$.  As usual, 
$\Leb{2}(\Omega)$ denotes the Hilbert space of real-valued functions on 
$\Omega$ that are measurable and square-integrable with respect to the Lebesgue 
measure of $\Rd$ and $\Hil{1}(\Omega)$ is the Hilbert space of all functions 
that, together with their distributional gradients, are in $\Leb{2}(\Omega)$.

Given $k\in\N$ with $k\leq d$, a set $\Simplex$ is a $k$-simplex in $\Rd$ if it 
is the convex hull of $k+1$ points $a_0,\dots,a_k\in\Rd$ that do not lie
on a plane of dimension $k-1$.  The set of extreme points of a convex set $C$ 
is denoted by $\Vertices C$.  For example, there holds 
$\Vertices\Simplex = \{a_0,\dots,a_k\}$ in the definition of $k$-simplex.  A 
$m$-simplex $\Face$ with $m\in\{1,\dots,k\}$ is a $m$-face of $\Simplex$ if 
$\Vertices\Face \subset \Vertices\Simplex$. By convention, a vertex is a 
$0$-face.  As usual, $h_\Simplex$ denotes the diameter of $\Simplex$, while 
$\rho_\Simplex$ stands for the diameter of the largest ball in $\Simplex$. The 
boundary of any $d$-simplex $\Simplex$ in $\Rd$ can be represented locally by a 
Lipschitz function.  Hence the trace operator 
\begin{equation}
\label{TraceOp}
 (\cdot)_{|\partial\Simplex}:\Hil{1}(\Simplex)\to\Leb{2}(\partial\Simplex)
\end{equation}
is well-defined.  Hereafter $\Leb{2}(\partial\Simplex)$ stands for the
Hilbert space of all real-valued functions on $\partial\Simplex$ that
are measurable and square-integrable with respect to the
$(d-1)$-dimensional Hausdorff measure.


Assume that $\mesh$ is a conforming simplicial mesh of $\Omega$.  More
precisely, $\mesh$ is a finite sequence of $d$-simplices in $\Rd$ such
that
\begin{gather}
\label{A:conforming_mesh}
 \overline\Omega = \bigcup_{\elm\in\mesh} \elm
\qquad\text{and}\qquad
 \forall \elm,\elm'\in\mesh
\quad
 \Vertices (\elm\cap\elm') \subset \Vertices\elm \cap \Vertices\elm'.
\end{gather}
In \S\ref{S:ConformityImpact} below we make an additional
assumption on the mesh $\mesh$.  The second condition in
\eqref{A:conforming_mesh} has two implications.  First, it ensures
that $\mesh$ is a non-overlapping covering in that there holds
\[
 \vol{\Omega}
 =
 \sum_{\elm\in\mesh} \vol{\elm}
\]
for the Lebesgue measure $\vol{\cdot}$ in $\Rd$.  Second, it entails
that $\mesh$ is conforming or face-to-face, i.e.\ for any two
`elements' $\elm,\elm'\in\mesh$ the intersection $\elm\cap\elm'$ is a
$k$-face of both $d$-simplices $\elm$ and $\elm'$ for some
$k\in\{0,\dots,d\}$. 

The existence of a conforming simplicial mesh implies that $\Omega$ is
bounded and some regularity for its boundary $\partial\Omega$.  In
particular, cusps are outruled but domains like the slit domain
$\{x=(x_1,x_2)\in\R^2\mid \max\{|x_1|,|x_2|\}<1, x_2 \neq 0 \text{ or
} x_1<0 \}$ are allowed.  For the latter example, the usual definition
of the trace operator $\Hil{1}(\Omega)\to\Leb{2}(\partial\Omega)$ does
not apply as the boundary is not locally a graph of a function.
Nevertheless, exploiting \eqref{TraceOp}, we can define that a 
function $v\in\Hil{1}(\Omega)$ equals a function $g\in\Leb2(\partial\Omega)$ on 
the boundary by
\[
 v_{|\partial\Omega}=g
\quad\text{:$\iff$}\quad
 \forall\elm\in\mesh\ v_{|\partial\Omega\cap\partial\elm} = 
   g_{|\partial\Omega\cap\partial\elm}.
\]

\medskip As mentioned in the introduction \S\ref{S:Intro}, the space
\[
 \Splines^{\pdeg,0}_0(\mesh)
 \defas
 \{ V:\overline\Omega\to\R \mid
   \forall\elm\in\mesh\ V_{|\elm}\in\Poly{\pdeg}{\elm},\ 
    V\in C^0(\overline\Omega),\ 
   V_{|\partial\Omega}=0 \}
\]
with $\pdeg\in\N$ may be used when approximating functions in
\begin{equation}
\label{D:H10}
 H^1_0(\Omega)
 \defas
 \{v\in\Leb{2}(\Omega) \mid
  \grd v\in\Leb{2}(\Omega),\ 
  v_{|\partial\Omega} = 0 \}.
\end{equation}
The first property in the definition of $\Splines^{\pdeg,0}_0(\mesh)$
determines the basic nature of the approximants: their piecewise
structure and that each one can be identified with a finite number of
parameters.  The role of the other two properties, which constrain this basic 
nature, is clarified by the following proposition in terms of the space
\[
 \Splines^{\pdeg,{-1}}(\mesh)
 \defas
 \{ V:\overline\Omega\to\R \mid
   \forall\elm\in\mesh\ V_{|\elm}\in\Poly{\pdeg}{\elm} \}
\]
of all functions that are piecewise polynomial over $\mesh$ and the
space
\[
 \Splines^{\pdeg,0}(\mesh)
 \defas
 \{ V\in\Splines^{\pdeg,{-1}} \mid
   V\in C^0(\overline\Omega) \}
\]
of all functions that are in addition continuous.

\begin{prop}[Characterization of $H^1$- and $H^1_0$-conformity]
\label{P:conforming}
  A piecewise polynomial function $V\in\Splines^{\pdeg,{-1}}(\mesh)$ is
  in $\Hil{1}(\Omega)$ if and only if it is continuous in
  $\overline\Omega$.  Moreover, a continuous piecewise polynomial
  function $V\in\Splines^{\pdeg,0}(\mesh)$ is in $\Hil{1}_0(\Omega)$
  if and only if $V_{|\partial\Omega}=0$.
\end{prop}

\begin{proof}
  The first equivalence is a consequence of
  \cite[Chapter II, Theorem 5.1]{Braess:01}, while then the second one
  immediately follows from definition \eqref{D:H10}.
\end{proof}

Hence the requirements $V\in C^0(\overline\Omega)$ and $V_{|\partial\Omega}=0$ 
are sufficient and necessary for conformity of the approximants.  Clearly, the 
first requirement $V\in C^0(\overline\Omega)$ is independent from the 
considered boundary condition $v_{|\partial\Omega}$; see also Corollaries 
\ref{C:Dirichlet} and \ref{C:Neumann} below. 

%
\medskip Next, we recall the Lagrange basis of $\Splines^{\pdeg,0}(\mesh)$, 
$\pdeg\in\N$.  Using multi-index notation, the Lagrange nodes of order $\pdeg$ 
of a $k$-simplex $\Simplex=\Conv\{a_0,\dots,a_k\}$ in $\Rd$ are given by
\[
 \LgrNodes{\pdeg}(\Simplex)
 \defas
 \left\{
  \frac1\pdeg \sum_{i=1}^k \alpha_i a_i \mid
   \alpha=(\alpha_0,\dots,\alpha_k)\in\N_0^{k+1}, \
   \MIdxOrd{\alpha}=\pdeg
  \right\}.
\]
These nodes determine $\Poly{\pdeg}{\Simplex}$ in that there holds
\[
 P_{|\LgrNodes{\pdeg}(\Simplex)}=Q_{|\LgrNodes{\pdeg}(\Simplex)}
\implies
 P=Q \text{ on }\Simplex
\]
for any $P,Q\in\Poly{\pdeg}{\Simplex}$.  Moreover, they are compatible with the 
dimensional structure of a $d$-simplex: if $\Face$ is a $k$-face of $\Simplex$, 
then
\begin{equation}
\label{LagrangeCompatability}
 \LgrNodes{\pdeg}(\Simplex)\cap\Face=\LgrNodes{\pdeg}(\Face).
\end{equation}
The Lagrange nodes of the space $\Splines^{\pdeg,0}(\mesh)$ are
\[
 \LgrNodes{\pdeg}(\mesh)
 \defas
 \bigcup_{\elm\in\mesh} \LgrNodes{\pdeg}(\elm);
\]
notice that a node $z\in\LgrNodes{\pdeg}(\mesh)$ may be shared by
several elements $K\in\mesh$.

Using the conformity of $\mesh$ and \eqref{LagrangeCompatability}, the Lagrange 
basis $\{\LgrBFct_z\}_{z\in\LgrNodes{\pdeg}(\mesh)}$ of 
$\Splines^{\pdeg,0}(\mesh)$ is characterized by
\begin{equation}
\label{LgrBasis}
 \LgrBFct_z\in\Splines^{\pdeg,0}(\mesh)
\quad\text{and}\quad
 \forall y\in\LgrNodes\pdeg(\mesh)\ 
 \LgrBFct_z(y) = \delta_{yz}
\end{equation}
and, for any $V\in\Splines^{\pdeg,0}(\mesh)$, provides the representation
\begin{equation*}
  V = \sum_{z\in\LgrNodes{\pdeg}(\mesh)} V(z) \LgrBFct_z
\end{equation*}
in terms of its nodal values $\{V(z)\}_{z\in\LgrNodes{\pdeg}(\mesh)}$. Each 
basis function $\LgrBFct_z$ is locally supported.  More precisely, the support
$\supp\LgrBFct_z\defas\overline{\{x\in\Omega\mid\LgrBFct_z\neq0\}}$, which 
corresponds to the domain of influence of the degree of freedom
$V \mapsto V(z)$, verifies
\[
 \supp\LgrBFct_z
 =
 \omega_z\
 \defas
 \bigcup_{\elm\in\mesh:\elm\ni z} \elm.
\]
It is useful to determine the scaling behavior of norms of the Lagrange basis 
functions.  To this end, we associate nodes of a given element $\elm\in\mesh$ 
to the ones of the reference simplex $\Simplexd=\Conv\{0,e_1,\dots,e_d\}$ in 
$\Rd$ in the following manner. Given $z\in\LgrNodes{\pdeg}(\elm)$, write 
$z=\sum_{i=0}^d\lambda_i a_i$ as a convex combination of the vertices of 
$\elm$, rearrange the coefficients in decreasing order such that
$\lambda_0^*\geq\dots\geq\lambda_d^*$ and set
\begin{equation}
\label{RefNode}
 \Hat{z}\defas\sum_{i=1}^d \lambda_i^* e_i
 \in
 \LgrNodes{\pdeg}(\Simplexd).
\end{equation}
Notice that, although the rearrangement may be not unique, $\Hat z$ is
well-defined and does not depend on the enumeration of the vertices of $\elm$.  
Fixed one node $z\in\LgrNodes{\pdeg}(\elm)$, there exists a (not necessarily 
unique) affine mapping $A:\Rd\to\Rd$ such that $A(\Simplexd)=\elm$ and $A(\Hat 
z)=z$.  Hence the transformation rule and $\vol{\Simplexd}=1/d!$ leads to
\begin{equation}
\label{LebnormBFct2}
 \Lebnorm{\LgrBFct_z}2{\elm}
 =
 \sqrt{d!}\vol{\elm}^{\frac12} \Lebnorm{\Hat\LgrBFct_{\Hat z}}2{\Simplexd}
\end{equation}
where $\vol{\elm}$ is the $d$-dimensional Lebesgue measure of $\elm$ and 
$\Hat\LgrBFct_{\Hat z}\in\Poly{\pdeg}{\Simplexd}$ is the polynomial that is 1 
at $\Hat{z}$ and vanishes at the other Lagrange nodes 
of $\Simplexd$.

Furthermore, it is useful to extend the linear functionals
\begin{equation}
\label{PointEVal}
 \delta_z : 
 \Splines^{\pdeg,0}(\mesh) \ni V
 \mapsto
 V(z) \in \R
\end{equation}
to functions in $\Hil{1}(\Omega)$.  To this end, we invoke the construction of 
L.\ R.\ Scott and S.\ Zhang \cite{Scott.Zhang:90}, which exploits the trace 
theorem \eqref{TraceOp}.  Denote the set of all $(d-1)$-faces of $\mesh$ by 
$\Faces(\mesh)$ and fix a $(d-1)$-dimensional face $F\in\Faces(\mesh)$.  Let 
$\{\DBFct^F_z\}_{z\in\LgrNodes{\pdeg}(\Face)}$ be the $\Leb{2}(\Face)$-Riesz 
representations in $\Poly{\pdeg}{\Face}$ of the functionals 
$\{\delta_z^\Face\}_{z\in\LgrNodes{\pdeg}(\Face)}$ given by 
$\delta_z^\Face(P)=P(z)$ for all $P\in\Poly{\pdeg}{\Face}$.  Since the latter 
functionals are the basis of $\Poly{\pdeg}{\Face}'$ that is dual to 
$\{\LgrBFct_z{}_{|\Face}\}_{z\in\LgrNodes{\pdeg}(\Face)}$, this amounts to 
defining $\{\DBFct_z^\Face\}_{z\in\LgrNodes{\pdeg}(\Face)}$ by
\begin{equation}
\label{DBFct}
 \DBFct_z^\Face\in\Poly{\pdeg}{\Face}
\quad\text{and}\quad
 \forall y\in\LgrNodes{\pdeg}(\Face) \ 
  \int_F \LgrBFct_y \DBFct_z^\Face = \delta_{yz}.
\end{equation}
In order to state the counterpart of \eqref{LebnormBFct2}, use
\eqref{RefNode} to associate a given node
$z\in\LgrNodes{\pdeg}(\Face)$ of some face $\Face\in\Faces(\mesh)$ to
a node of the reference simplex $\Simplex_{d-1}$ in $\R^{d-1}$.  Upon
identifying $\R^{d-1}$ with $\R^{d-1}\times\{0\}\subset\Rd$, this
gives the same point as applying \eqref{RefNode} for some element
$\elm$ containing $\Face$.  Again the transformation rule leads to
\begin{equation}
\label{LebnormDBFct2} 
 \Lebnorm{\DBFct_z}2{\Face}
 =
 \frac{1}{\sqrt{(d-1)!}}\area{\Face}^{-\frac12}
 \Lebnorm{\Hat\DBFct_{\Hat z}}2{\Simplex_{d-1}},
\end{equation}
where $\area{\Face}$ is the $(d-1)$-dimensional Hausdorff measure of
$\Face$ and $\Hat\DBFct_{\Hat z}$ the $L^2(\Simplex_{d-1})$-dual basis
function corresponding to $\Hat\LgrBFct_{\Hat z}$.  Since
$\LgrBFct_y{}_{|\Face}=0$ for each
$y\in\LgrNodes{\pdeg}(\mesh)\setminus\LgrNodes{\pdeg}(\Face)$, the
orthogonality condition in \eqref{DBFct} extends to
\[
 \forall y\in\LgrNodes{\pdeg}(\mesh)
\quad 
 \int_F \LgrBFct_y \DBFct_z^\Face = \delta_{yz}.
\]
Given $z\in\LgrNodes{\pdeg}(\Face)$, this yields
\[
 \forall V\in\Splines^{\pdeg,0}(\mesh)
\quad 
 V(z) = \NodalVar_{z;\Face}(V)
\]
where $\NodalVar_{z;\Face}$ is defined on $\Hil{1}(\Omega)$ by
\begin{equation}
\label{ScottZhangFct}
 \NodalVar_{z;\Face}(v)
 \defas
 \int_F v \DBFct_z^\Face.
\end{equation}
Thus, given a node $z\in\LgrNodes{\pdeg}(\mesh)$ of
$\Splines^{\pdeg,0}(\mesh)$, any $\NodalVar_{z;\Face}$ with $\Face \ni
z$ extends the linear functional \eqref{PointEVal} to
$\Hil{1}(\Omega)$. Finally, we note that if $z\in\partial\Omega$,
then the implication
\[
 v\in\Hil{1}_0(\Omega)
 \implies
 N_{z;F}(v) = 0
\]
is ensured whenever there hold $F\ni z$ and $F\subset\partial\Omega$ is a 
boundary face.

\section{Conformity and approximation error}
\label{S:ConformityImpact}
Proposition \ref{P:conforming} determines conditions that characterize when 
piecewise polynomial functions are conforming.  The main goal of this section is 
to analyse the impact of these conditions on the error when approximating the 
gradient of a function.

\subsection{Global and local best errors}
\label{S:best-errors}
%
%
We first provide a notion that measures the possible downgrading resulting from
conformity.  In order to take into account boundary conditions, we consider the 
following setting for the space $X$ of target functions and the approximants 
$\Splines(\mesh)$ over a mesh $\mesh$.  Assume that $X$ and $\Splines(\mesh)$ 
are, respectively, closed affine subspaces of $\Hil1(\Omega)$ and 
$\Splines^{\pdeg,0}(\mesh)$ with $\pdeg\geq1$ and that the $\Hil1$-seminorm
\begin{equation}
\label{H1seminorm}
 \error{w}{\Omega}
 \defas
 \Lebnorm{|\grd w|}2\Omega
 =
 \left( \int_\Omega |\grd w|^2 \right)^{\frac12}.
\end{equation}
is definite on $X-\Splines(\mesh)$.   The setting $X=\Hil1_0(\Omega)$ and
$\Splines(\mesh)=\Splines^{\pdeg,0}_0(\mesh)$ in the introduction 
\S\ref{S:Intro} is an example.

The best (possible) error of approximating $v\in X$ is then given by
\[
 E(v,\mesh)
 \defas
 E\big( v,\Splines(\mesh) \big)
 \defas
 \inf_{V\in\Splines(\mesh)} \error{v-V}{\Omega}
\]
and $E(v,\mesh)=0$ implies $v\in\Splines(\mesh)$. Thanks to, e.g., the
Projection Theorem in Hilbert spaces, there exists a unique best
approximation $V_\mesh$ such that
\[
 \error{v-V_\mesh}{\Omega}
 =
 E(v,\mesh).
\]
If $S(\mesh)$ is a linear space, this is  equivalent to
\[
 \forall W\in\Splines(\mesh)
\quad
 \int_\Omega \grd V_\mesh \cdot \grd W
 =
  \int_\Omega \grd v \cdot \grd W.
\]
and $V_\mesh$ is called the Ritz projection of $v$ onto $\Splines(\mesh)$.

Similarly, on each single element $\elm\in\mesh$, the best error is
given by
\[
 e(v,\elm)
 \defas
 e\big( v,\Poly{\pdeg}{\elm} \big)
 \defas
 \inf_{P\in\Poly{\pdeg}{\elm}} \error{v-P}{\elm},
\]
which depends only on the local gradient $(\grd v)_{|\elm}$ of the
target function and the shape functions associated with $\elm$.
Applying the Projection Theorem in the Hilbert space
$\Hil{1}(\elm)/\R$, we see that here best approximations also exist
but are only unique up to a constant.  Let $P_\elm$ be the best
approximation that has the same mean value on $\elm$ as the target
function $v$.  In other words, $P_\elm$ is characterized by
\begin{equation}
\label{Pelm}
 P_\elm\in\Poly{\pdeg}{\elm},
\quad
 \int_\elm P_\elm = \int_\elm v
\quad\text{and}\quad
 \error{v-P_\elm}{\elm}=e(v,\elm),
\end{equation}
the latter being equivalent to
\begin{equation}
\label{Pelm:ort}
 \forall Q\in\Poly{\pdeg}{\elm}
\quad
 \int_\elm \grd P_\elm \cdot \grd Q
 =
 \int_\elm \grd v \cdot \grd Q.
\end{equation}

Since the local best errors cannot be overtaken by any global approximation, 
one may expect that an appropriate `sum' of them provides a lower bound for 
corresponding global best errors.  In fact, if $\Splines'$ is any subspace of 
$\Splines^{\pdeg,-1}(\mesh)$, one readily verifies
\begin{align}
\label{local<=global}
 \left[
  \sum_{\elm\in\mesh} e(v,\elm)^2
 \right]^{\frac12}
 \leq
  E(v,\Splines'),
\end{align}
interpreting the right-hand side in a broken manner if necessary. Notice that 
there holds equality for $\Splines'=\Splines^{\pdeg,-1}(\mesh)$, while for 
$\Splines'=\Splines(\mesh)\subset\Splines^{\pdeg,0}(\mesh)$ the question arises 
if the requirement of continuity entails some downgrading of the approximation 
quality: the local approximants $P_\elm$, $\elm\in\mesh$, on the left-hand side 
are decoupled, while their counterparts $V_\mesh{}_{|\elm}$, $\elm\in\mesh$, on 
the right-hand side are coupled and so constrained.

In order to measure the possible downgrading, consider the inequality opposite 
to \eqref{local<=global} and denote by $\DcConst(\mesh)$ the smallest constant 
$C$ such that
\begin{equation}
\label{global<=Clocal}
 \forall v\in X
\quad
 E\big( v,\Splines(\mesh) \big)
 \leq
 C  \left[
  \sum_{\elm\in\mesh} e(v,\elm)^2
 \right]^{\frac12}
\end{equation}
is valid; if there is no such constant $C$, set $\DcConst(\mesh) = \infty$.  We 
refer to $\DcConst(\mesh)$ as the decoupling coefficient of $\Splines(\mesh)$.  
If the decoupling coefficient is big, or even $\infty$, requiring continuity
entails that the local approximation potential is not well exploited, at least
for some target functions.  If it is moderate, or even 1, dispensing with 
continuity does not improve the approximation quality in a substantial manner.

It is instructive to consider, for a moment, \eqref{global<=Clocal} with
$X=\Leb2(\Omega)$ and to replace the $\Hil1$-seminorm by the $\Leb{2}$-norm.
Then a function from
$\Splines^{\pdeg,-1}(\mesh)\setminus\Splines^{\pdeg,0}(\mesh)$ is an
admissible target function and we immediately obtain that there holds
$\DcConst(\mesh)=\infty$ in this case.

The following lemma introduces the key property of the $\Hil1$-seminorm
that ensures a finite decoupling coefficient.
\begin{lem}[Trace and error norm]
\label{L:TraceAndError}
Let $\Face$ be a $(d-1)$-face of a $d$-simplex $\Simplex$.  For any
$w\in\Hil1(\Simplex)$ with $\int_\Simplex w = 0$, there holds
\[
 \Lebnorm{w}2\Face
 \leq
 \CTr \left(
   \frac{h_\Simplex \area{\Face}}{\vol{\Simplex}}
  \right)^{\frac12}
  h_\Simplex^{\frac12} \Lebnorm{\grd w}2\Simplex,
\]
where $\CTr\defas\sqrt{C_P(C_P+2/d)}$ and $C_P$ is the optimal Poincar\'e 
constant for $d$-simplices.  The quotient between the parathensis is bounded in 
terms of the shape coefficient $h_\Simplex/\rho_\Simplex$ of $\Simplex$.
\end{lem}

The classical result of L.\ E.\ Payne and H.\ F.\ Weinberger
\cite{Payne.Weinberger:60}, see also M.\ Bebendorf \cite{Bebendorf:03}, 
ensures $C_P\leq1/\pi$.  In the case $d=2$, R.\ S.\ Laugesen and B.\ A.\ 
Siudeja \cite{Laugesen.Siudeja:10} show $C_P=1/{j_{1,1}}$ where 
$j_{1,1}\approx 3.8317$ denotes the first positive root of the Bessel function 
$J_1$. 

\begin{proof}
%
Corollary 4.5 and Remark 4.6 of \cite{Veeser.Verfuerth:09} imply the
trace inequality
\begin{equation*}
 \frac1{\area{\Face}} \Lebnorm{w}2\Face^2
 \leq
 \frac1{\vol{\Simplex}} \Lebnorm{w}2\Simplex^2
 +
 \frac{2h_\Simplex}{d\vol{\Simplex}} \Lebnorm{w}2\Simplex
  \Lebnorm{\grd w}2\Simplex.
\end{equation*}
We thus obtain the claimed inequality for $w$ by inserting the Poincar\'e 
inequality
\begin{equation*}
 \Lebnorm{w}2\Simplex
 \leq
 C_P h_\Simplex \Lebnorm{\grd w}2\Simplex,
\end{equation*}
which applies thanks to $\int_\Simplex w = 0$.

In order to bound the quotient between the parathensis, observe that 
Cavallieri's principle yields $\vol{\Simplex} = (h^\perp_\Face\area{\Face})/d$, 
where $h^\perp_\Face$ denotes the height of $\Simplex$ over $\Face$.  Hence we 
obtain
\[
 \frac{h_\Simplex \area{\Face}}{\vol{\Simplex}}
 =
 d\frac{h_\Simplex}{h^\perp_F}
 \leq
 d\frac{h_\Simplex}{\rho_\Simplex}
\]
with the help of the inequality $\rho_\Simplex \leq h^\perp_\Face$.
\end{proof}
%
The significance of Lemma \ref{L:TraceAndError} lies in the following 
observations.
If the intersection $\Face$ of two elements $K_1,K_2\in\mesh$ is a common 
$(d-1)$-face, then the condition $V\in C^0(\overline\Omega)$ requires that the 
traces $V_{|\partial\elm_1}$ and $V_{|\partial\elm_2}$ coincide on $\Face$.  
The local best approximation $P_{\elm_1}$ and $P_{\elm_2}$ are close to this 
property in that, thanks to Lemma \ref{L:TraceAndError}, their properly 
measured difference is bounded in terms of the local best errors:
\begin{align*}
 h_\Face^{-\frac12} \Lebnorm{P_{\elm_1}-P_{\elm_2}}2\Face
 &\leq
 h_\Face^{-\frac12} \Lebnorm{P_{\elm_1}-v}2\Face
  + 
  h_\Face^{-\frac12} \Lebnorm{v-P_{\elm_2}}2\Face
\\
 &\leq
 C \big[ e(v,\elm_1) + e(v,\elm_2) \big],
\end{align*} 
where $h_\Face\defas\diam\Face$ and $C$ depends on $d$ and on the shape 
coefficients of $\elm_1$ and $\elm_2$.  Consequently, the trace $V_{|\Face}$ 
can be defined by $P_{K_1}$, $P_{K_2}$ or a mixture of both without 
substantially downgrading the approximation potential of the shape functions of 
the two elements $\elm_1$ and $\elm_2$.  The same remark applies to near best 
approximations in place of $P_{K_1}$ and $P_{K_2}$.

Similary, Lemma \ref{L:TraceAndError} implies that, thanks to 
$v_{|\partial\Omega}=0$, properly measured traces on $\partial\Omega$ of local 
best approximations are bounded again in terms of local best errors or, in 
other words, almost vanish.  Consequently, enforcing vanishing boundary values 
will not downgrade the approximation quality.

\subsection{Interpolation}
%
In order to show that the decoupling coefficient $\DcConst(\mesh)$ is finite, 
we define
\[
 \Ip:\Hil{1}(\Omega)\to\Splines^{\pdeg,0}(\mesh)
\]
such that
\begin{equation}
\label{Ip;BdryVals}
 v\in X
 \implies
 \Ip v\in\Splines(\mesh)
\end{equation}
and
\begin{equation}
\label{GDcplng}
 \error{v-\Ip v}{\Omega}
 \leq
 C \left[
  \sum_{\elm\in\mesh} e(v,\elm)^2
 \right]^{\frac12}
\end{equation}
for some constant $C$, independent of $v$.  Notice that the latter
property requires that $\Ip$ is a projection whenever $\Splines(\mesh)
\subset X$ and stable with respect to \eqref{H1seminorm}.

\begin{subequations}
\label{DefIp}
\medskip Using the Lagrange basis \eqref{LgrBasis} of
$\Splines^{\pdeg,0}(\mesh)$, we can write
\[
 \Pi v
 =
 \sum_{z\in\LgrNodes{\pdeg}(\mesh)} \Ip_z v \, \LgrBFct_z
\]
and defining $\Ip$ amounts to choosing suitable linear functionals
$\Ip_z\in\Hil{-1}(\Omega)$, $z\in\LgrNodes{\pdeg}(\mesh)$, for the nodal
values.  For this purpose, it is convenient to introduce the following notion.
A node $z\in\LgrNodes{\pdeg}(\mesh)$ is called unconstrained in
$\Splines(\mesh)$ if and only if $\supp\LgrBFct_z$ is contained in one element
of $\mesh$ and there holds $\spn \LgrBFct_z\subset\Splines(\mesh)$.  Extending
functions on elements by zero, this is equivalent to requiring that there
holds $\spn \LgrBFct_z{}_{|\elm}\subset\Splines(\mesh)$ for any element
$\elm\in\mesh$.  Unotherwise $z$ is called constrained and we write
$z\in\Cnstr$.

Fix an arbitrary node $z\in\LgrNodes{\pdeg}(\mesh)$. If $z\not\in\Cnstr$, then
$\Ip_zv$ affects only the local error $\error{v - \Ip v}{\elm}$ on that
element.  In this case, set
\begin{equation}
\label{DefIp:unconstrained}
 \Ip_z v \defas P_\elm(z),
\end{equation}
where $P_\elm$ is given by \eqref{Pelm}.

Whereas, if $z\in\Cnstr$, then $\Ip_zv$ has to deviate from 
\eqref{DefIp:unconstrained} for at least one other element or has to assume a
prescribed value.  In order to meet both issues,  we employ the Scott-Zhang 
functionals \eqref{ScottZhangFct}.  To this end, fix some face 
$\Face_z\in\Faces(\mesh)$ containing $z$; this choice may be subject to
further conditions for certain examples of $\Splines(\mesh)$.  We then set
\begin{equation}
\label{DefIp:constrained}
 \Ip_z v
 \defas
 \NodalVar_{z;F_z}(v).
\end{equation}

We illustrate this definition in the setting of the introduction \S\ref{S:Intro}
where $X=\Hil1_0(\Omega)$ and $\Splines(\mesh)=\Splines^{\pdeg,0}_0(\mesh)$.
Here there holds
\[
 \Cnstr = \LgrNodes{\pdeg}(\mesh) \cap \Skeleton
\quad\text{with}\quad
 \Skeleton
 \defas
 \bigcup_{\Face\in\Faces(\mesh)} \Face
\]
and, if $z\in\Cnstr$, we addditionally require that 
$\Face_z\subset\partial\Omega$ whenever $z\in\partial\Omega$ lies on the
boundary.  This readily ensures \eqref{Ip;BdryVals}.  In Remark \ref{R:Ip}
below we further discuss the constuction of $\Ip$, comparing it with existing
interpolation operators and indicating alternatives.  In particular, we shall
see that, irrespective of the choice of $\Face_z$, the definition
\eqref{DefIp:unconstrained} is `near to the best \eqref{DefIp:constrained}' in
a suitable sense.
\end{subequations}

Notice that, on one hand, the face in \eqref{DefIp:constrained} is
linked to an arbitrary element of $\supp\LgrBFct_z$ only through the
node $z$ and, on the other hand, traces of $\Hil1$-functions are
well-defined only on at least $(d-1)$-dimensional faces.  The
following property of the mesh therefore appears to be essential for
local near best approximation properties of $\Ip$.  A star
\[
 \omega_z = \bigcup\{\elm\in\mesh : \elm\ni z\},
\qquad
 z\in\LgrNodes{\pdeg}(\mesh),
\]
is called $(d-1)$-face-connected if for any element $\elm$ and $(d-1)$-face
$\Face$ containing $z$ there exists a sequence $(\elm_i)_{i=1}^n$ such
that
\begin{itemize}
\item any $\elm_i$, $i=0,\dots,n$, is an element of the star,
\item any intersection $\elm_i\cap\elm_{i+1}$,
  $i=0,\dots,n-1$, is a $(d-1)$-face of the star,
\item $\elm_0$ contains $\Face_z$ and $\elm_n=\elm$.
\end{itemize}
This property holds for stars of interior nodes $z\in\Omega$.  The case of a 
boundary node $z\in\partial\Omega$ is more subtle.  It holds whenever the 
boundary is a Lipschitz graph in a sufficiently big neighborhood of $z$.  
Noteworthy, the latter property is not necessary; see, e.g., the slit domain 
and Figure \ref{F:edge-connectedness} (left).  Roughly speaking, it will be 
verified also in similar cases where the boundary can be accessed from both 
sides if that part is at least $(d-1)$-dimensional, otherwise it will not hold. 
Figure \ref{F:edge-connectedness} (right) illustrates such a counterexample 
for $d=2$.  
\begin{figure}[h]
 \centering
 \includegraphics[bb=176 37 650 505,
                  width=0.4\hsize,
                  height=20ex]{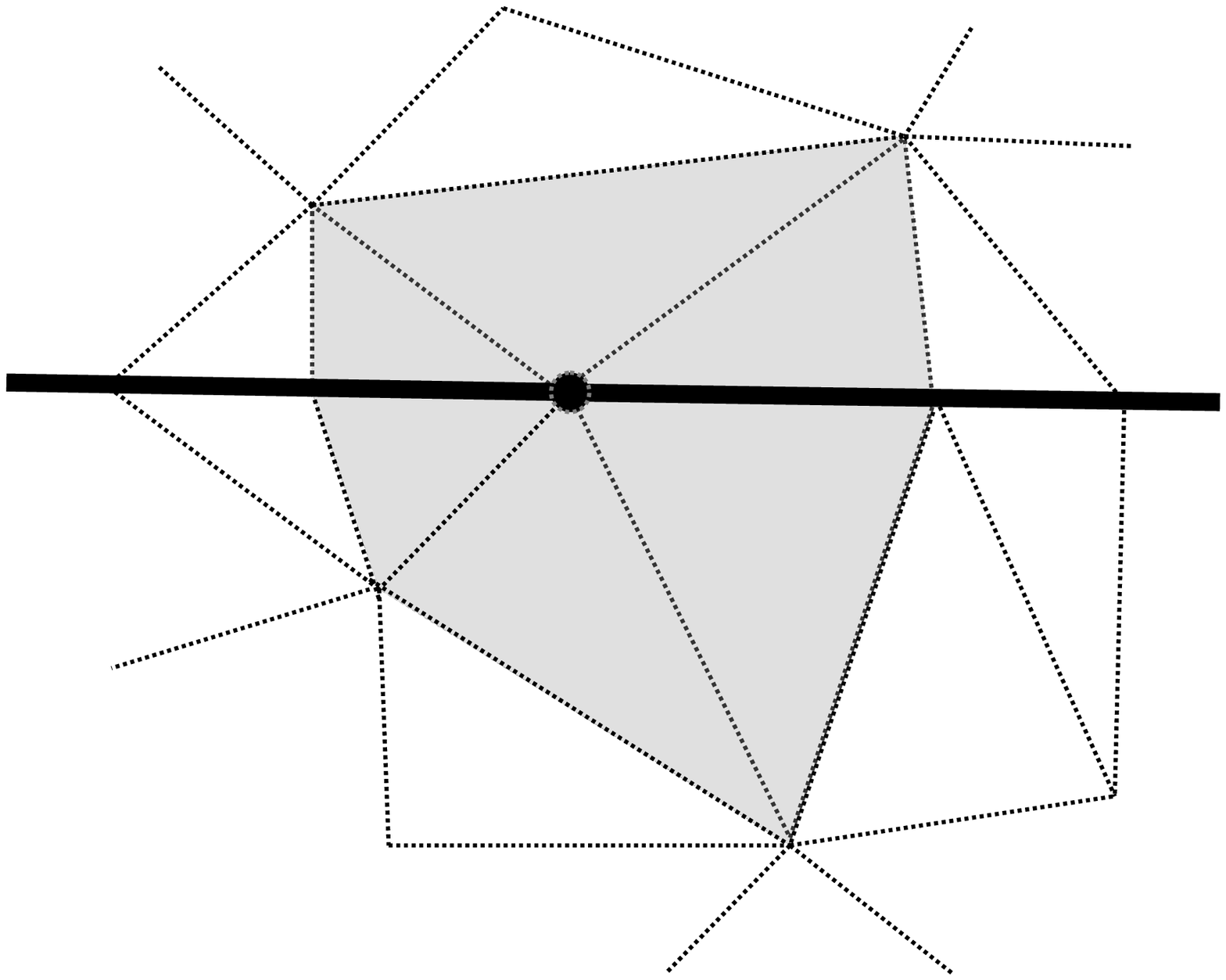}
\qquad\qquad
 \includegraphics[bb=176 37 650 505,
                  width=0.4\hsize,
                  height=20ex]{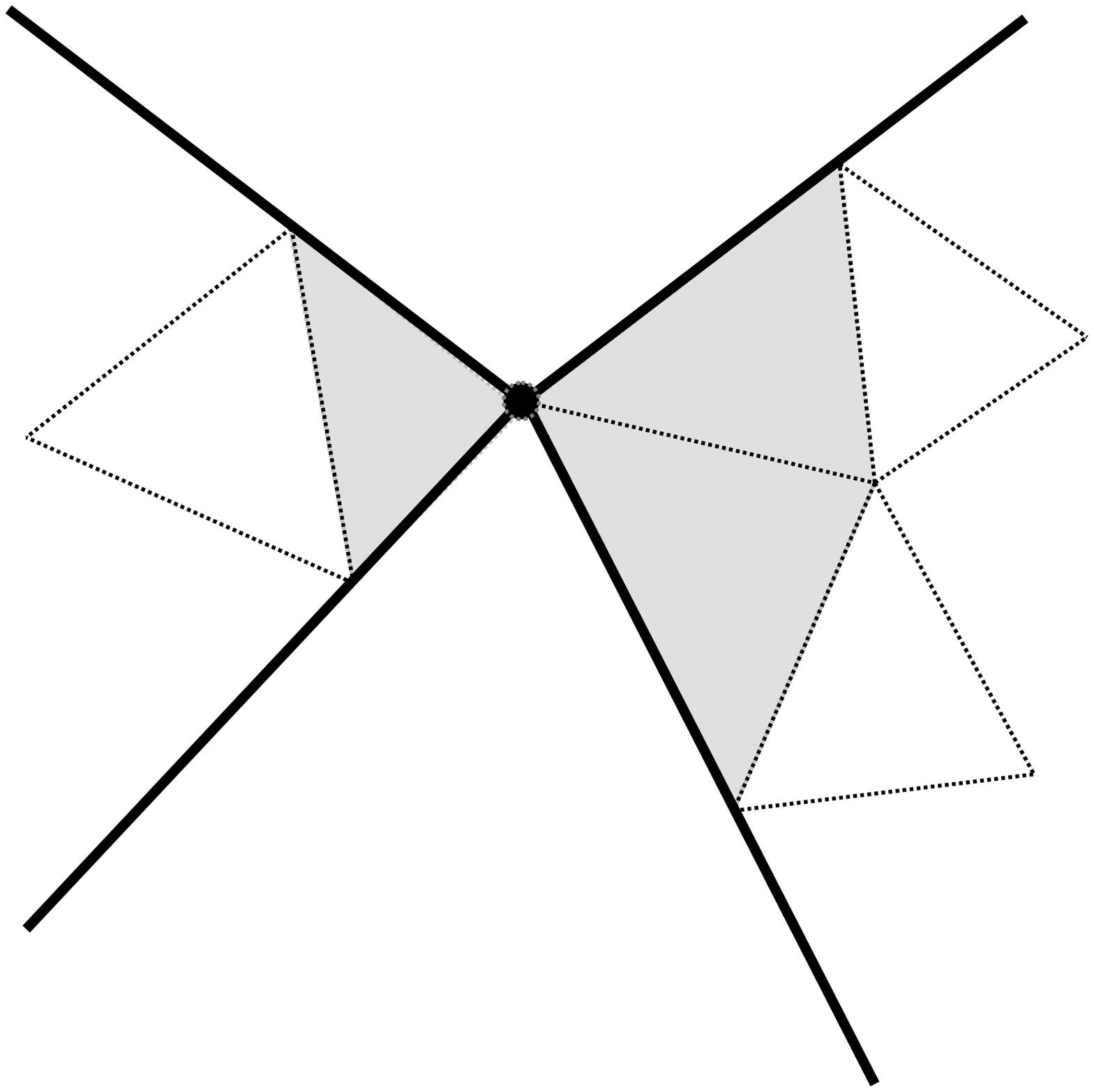}
 \caption{Edge-connectedness: two-dimensional stars (grey areas) at the 
boundary (thick lines), the left one is edge-connected, the right one is not.}
 \label{F:edge-connectedness}
\end{figure}
It is worth noting that, if a star is not $(d-1)$-face-connected, the $H^1$-norm 
is not strong enough to couple both sides across their common part of the 
boundary, suggesting that also corresponding elements should not be coupled in 
the finite element space.

\subsection{Local decoupling}
\label{S:ldecoupling}
%

The definition of $\Pi$ and the identity
\[
 \error{w}{\Omega}^2
 =
 \sum_{\elm\in\mesh}\error{w}{\elm}^2,
\]
suggest to show \eqref{GDcplng} by establishing local counterparts, involving
the patches
\[
 \omega_\elm
 \defas
 \bigcup \{\elm'\in\mesh:\elm'\cap\elm\neq\emptyset\}
\]
in $\mesh$.
\begin{thm}[Local decoupling]
\label{T:LDcplng}
Assume that the stars $\omega_z$, $z\in\LgrNodes{\pdeg}(\elm)$, associated with 
a given element $\elm\in\mesh$ are $(d-1)$-face-connected.  Then there exists a 
constant $\DcConst_\elm$ such that, for all $v\in\Hil1(\omega_\elm)$,
\begin{equation}
\label{LDcplng}
 \error{v - \Ip v}{\elm}^2
 \leq
 e(v,\elm)^2 
 +
 \DcConst_\elm^2
 \sum_{z\in\LgrNodes{\pdeg}(\elm)} \sum_{\elm'\ni z} e(v,\elm')^2,
\end{equation}
where the second sum is over all $\elm'\in\mesh$ containing $z$.  The constant 
$\DcConst_\elm$ depends on $d$, $\ell$, and the elements in $\omega_\elm$.
\end{thm}
\begin{proof}
We start by recalling the definition \eqref{Pelm} of the best approximation 
$P_\elm$ and exploit the orthogonality \eqref{Pelm:ort} to write
\[
 \error{v - \Ip v}{\elm}^2
 =
 e(v,\elm)^2 
 +
 \error{P_\elm - \Ip v}{\elm}^2,
\]
where the second square on the right hand side measures the deviation of $\Ip 
v$ of being locally optimal.  Thanks to \eqref{DefIp:unconstrained}, there holds
\[
 (P_\elm - \Ip v)_{|\elm}
 =
 \sum_{z\in\LgrNodes{\pdeg}(\elm)}
  \big[ P_\elm(z) - \Ip_z v \big] \LgrBFct_z{}_{|\elm}
 =
 \sum_{z\in\LgrNodes{\pdeg}(\elm)\cap\Cnstr}
  \big[ P_\elm(z) - \NodalVar_{z,\Face_z}(v) \big] \LgrBFct_z{}_{|\elm},
\]
which reveals that this deviation is entirely related to the requirements for 
conformity.  To bound it, we first apply the triangle inequality to obtain
\begin{equation}
\label{CplngErr}
 \error{P_\elm - \Ip v}{\elm}
 \leq
  \sum_{z\in\LgrNodes{\pdeg}(\elm)\cap\Cnstr}
  \big| P_\elm(z) - \NodalVar_{z,\Face_z}(v) \big|
   \Lebnorm{\grd\LgrBFct_z}2\elm.
\end{equation}
We proceed by bounding each term of the sum separately.  To this end,
fix any $z\in\LgrNodes{\pdeg}(\elm)\cap\Cnstr$ and, using that
$\omega_z$ is $(d-1)$-face-connected, choose a corresponding sequence
$(\elm_i)_{i=0}^n$ of elements connecting $\Face_z$ with $\elm$.  Set
$\Face_0\defas\Face_z$ and $\Face_i\defas\elm_{i-1}\cap\elm_i$ for
$i=1,\dots,n$.  For the sake of readibility, we sometimes replace
$\elm_i$ and $\Face_i$ by $i$.  Recalling the definition \eqref{DBFct}
of the local dual basis, we note
\[
 P_\elm(z)-\NodalVar_{z,\Face_z}(v)
 =
 \NodalVar_{z,n}(P_n)-\NodalVar_{z,0}(v)
\]
and
\[
 \forall i=0,\dots,n-1
\quad
 \NodalVar_{z,i}(P_i)
 =
 P_i(z)
 =
 \NodalVar_{z,i+1}(P_i),
\]
which, by telescopic expansion, leads to
\begin{equation}
\label{Telescopic}
 P_\elm(z)-\NodalVar_{z,\Face_z}(v)
 =
 \NodalVar_{z,0}(P_0-v)
 +
 \sum_{i=1}^n \NodalVar_{z,i}(P_i-P_{i-1}).
\end{equation}
Since $\Face_0\subset\elm_0$ and $\int_{\elm_0} (v-P_0) = 0$, one can
combine \eqref{LebnormDBFct2} and Lemma \ref{L:TraceAndError} to
derive
\begin{align*}
 |\NodalVar_{z,0}(v-P_0)|
 &=
 \left|
  \int_{\Face_0} (v-P_0) \DBFct^i_z
 \right|
 \leq
 \Lebnorm{v-P_0}2{\Face_0} \Lebnorm{\DBFct^0_z}2{\Face_0}
\\
 &\leq
 \frac{\CTr}{\sqrt{(d-1)!}}
 \Lebnorm{\Hat\DBFct_{\Hat z}}2{\Simplex_{d-1}}
 \frac{h_{0}}{\vol{\elm_0}^{\frac12}}
  e(v,\elm_0),
\end{align*}
where $\area{\Face_0}$ cancels out.  Similarly, for $i=1,\dots,n$, one
obtains
\begin{equation*}\begin{aligned}
 \NodalVar_{z,i}(P_i &- P_{i-1})
 \leq
 |\NodalVar_{z,i}(P_i - v)| + |\NodalVar_{z,i}(v - P_{i-1})|
\\
 &\leq
 \frac{\CTr}{\sqrt{(d-1)!}}
 \Lebnorm{\Hat\DBFct_{\Hat z}}2{\Simplex_{d-1}}
 \left[
  \frac{h_{i}}{\vol{\elm_i}^{\frac12}} e(v,\elm_i)
  +
  \frac{h_{i-1}}{\vol{\elm_{i-1}}^{\frac12}} e(v,\elm_{i-1})
 \right].
\end{aligned}\end{equation*}
Using these bounds in the telescopic expansion \eqref{Telescopic} gives
\[
 |P_\elm(z)-\NodalVar_{z,\Face_z}(v)|
 \leq
 \frac{2\CTr}{\sqrt{(d-1)!}}
 \Lebnorm{\Hat\DBFct_{\Hat z}}2{\Simplex_{d-1}}
 \sum_{i=0}^n 
  \frac{h_{i}}{\vol{\elm_i}^{\frac12}} e(v,\elm_i).
\]
In order to get independent of the specific choice of $\Face_z$, we
observe that the sequence $(\elm_i)_{i=0}^n$ does not allow a double
occurence of some element and that each element is a subset of
$\omega_z=\supp\LgrBFct_z$.  We therefore replace the preceding
inequality by
\begin{equation}\label{nbnval}
 |P_\elm(z)-\NodalVar_{z,\Face_z}(v)|
 \leq
 \frac{2\CTr}{\sqrt{(d-1)!}}
 \Lebnorm{\Hat\DBFct_{\Hat z}}2{\Simplex_{d-1}}
  \sum_{\elm'\subset\omega_z}
   \frac{h_{\elm'}}{\vol{\elm'}^{\frac12}}
   e(v,\elm'),
\end{equation}
where the sum is over all elements $\elm'\in\mesh$ with
$\elm'\subset\omega_z$.  Inserting this inequality and
\eqref{LebnormBFct2} into \eqref{CplngErr}, one arrives at
\begin{equation*}
 \error{P_\elm - \Ip v}{\elm}
 \leq
 2 \sqrt{d} \CTr
 \sum_{z\in\LgrNodes{\pdeg}(\elm)\cap\Cnstr}
 \sum_{\elm'\subset\omega_z}
  d_{\Hat z}
   \frac{h_{\elm'} \Lebnorm{\grd\LgrBFct_z}2\elm \vol{\elm}^{\frac12}}
   {\Lebnorm{\LgrBFct_z}2\elm \vol{\elm'}^{\frac12}}
   e(v,\elm')
\end{equation*}
with
\[
 d_{\Hat z}
 \defas
 \Lebnorm{\Hat\DBFct_{\Hat z}}2{\Simplex_{d-1}}
 \Lebnorm{\Hat\LgrBFct_{\Hat z}}2{\Simplexd}.
\]
Applying a Cauchy-Schwarz inequality, one obtains a sum of the squares
of the local best errors and the claimed inequality with
\begin{equation}
\label{locDcplngConst}
 \DcConst_\elm^2
 =
 4 d \CTr
 \sum_{z\in\LgrNodes{\pdeg}(\elm)\cap\Cnstr}
  d_{\Hat z}^2 \mu_z
  \frac{h_\elm^2 \Lebnorm{\grd\LgrBFct_z}{\elm}2^2}
   {\Lebnorm{\LgrBFct_z}{\elm}2^2}
\end{equation}
and
\[
 \mu_z
 \defas
 \frac{\vol{\elm}}{h_\elm^2}
 \sum_{\elm'\subset\omega_z}
  \frac{h_{\elm'}^2}{\vol{\elm'}}.
\] 
\end{proof}

Some remarks on the local decoupling constant $\DcConst_\elm$ in
\eqref{locDcplngConst} are in order.  Only $\mu_z$ and
$h_\elm \Lebnorm{\grd\LgrBFct_z}{\elm}2 /
\Lebnorm{\LgrBFct_z}{\elm}2$ depend on geometrical properties of the
elements in the patch $\omega_\elm$.  Both terms, independently, do
not change under rescalings of $\omega_\elm$.  While
$h_\elm \Lebnorm{\grd\LgrBFct_z}{\elm}2 /
\Lebnorm{\LgrBFct_z}{\elm}2$ corresponds to an inverse
estimate, $\mu_z$ multiplies the number of elements in the star around
$z$ with the mean value of the quotient
\[
 \frac{h_{\elm'}^2\vol{\elm}}{h_{\elm}^2\vol{\elm'}}.
\]
To conclude this subsection, let us consider the case where the patch
$\omega_\elm$ consists of elements whose shape coefficients are
uniformly bounded by $\sigma_\elm$:
\begin{equation}
\label{LocShapeReg}
 \forall \elm'\subset\omega_\elm
\quad
 \frac{h_{\elm'}}{\rho_{\elm'}}
 \leq
 \sigma_\elm.
\end{equation}
Then a standard scaling argument, see e.g.\ \cite[(4.5.3)]{Brenner.Scott:08} 
shows that
\[
 \frac{h_\elm \Lebnorm{\grd\LgrBFct_z}{\elm}2}
  {\Lebnorm{\LgrBFct_z}{\elm}2}
 \leq
 \sigma_\elm
 \frac{\Lebnorm{\grd\Hat\LgrBFct_{\Hat z}}{\Simplexd}2}
  {\Lebnorm{\Hat\LgrBFct_{\Hat z}}{\Simplexd}2}.
\]
Moreover, since the (solid) angles of the elements in $\omega_\elm$ are bounded 
away from 0 in terms of $\sigma_\elm$, the number of elements in each star of 
$\omega_K$ is bounded in terms of $\sigma_\elm$. Comparing elements having 
common faces, we thus obtain that the diameters and the volumes of the elements
in a star of $\omega_\elm$ are comparable up to $\sigma_\elm$.  Consequently, 
there holds $\mu_z\leq C(\sigma_\elm)$ for each $z\in\LgrNodes{\pdeg}(\elm)$ 
and the following lemma.

\begin{lem}[$\DcConst_\elm$ and shape regularity]
Under the assumption \eqref{LocShapeReg} there holds
\[
 \DcConst_\elm^2
 \leq
 4 d \CTr
 C(\sigma_\elm)
 \sum_{z\in\LgrNodes{\pdeg}(\elm)\cap\Cnstr}
  \Lebnorm{\Hat\DBFct_{\Hat z}}{\Simplex_{d-1}}2^2
  \Lebnorm{\grd\Hat\LgrBFct_{\Hat z}}{\Simplexd}2^2.
\]
\end{lem}
\subsection{Global decoupling and boundary conditions}
Summing up the inequalities of Theorem \ref{T:LDcplng}, we obtain our main 
result.

\begin{thm}[Decoupling of elements]
\label{T:elm-dec}
Assume that all stars of the mesh $\mesh$ are $(d-1)$-face-connected.  Then 
there exists a constant $C$ such that
\begin{equation*}
 \error{v-\Ip v}{\Omega}
 \leq
 C \left[
  \sum_{\elm\in\mesh} e(v,\elm)^2
 \right]^{\frac12}
\end{equation*}
for all $v\in\Hil1(\Omega)$.  The constant $C$ depends on the dimension $d$, 
the polynomial degree $\pdeg$ and the shape coefficient $\sigma_\mesh \defas 
\max_{\elm\in\mesh}  h_\elm/\rho_\elm$ of $\mesh$.
\end{thm}

\begin{proof}
We sum \eqref{LDcplng} over all $\elm\in\mesh$.  On the right-hand side, we hit 
a given element $\elm'\in\mesh$ at most $1+n_\pdeg N_\mesh$ times, where 
$n_\pdeg = \#\{z\in\LgrNodes{\pdeg}(\elm)\mid z\in\partial\elm\}$ indicates the 
number of boundary Lagrange nodes and 
\[
 N_\mesh
 \defas
 \max_{z\in\LgrNodes{\pdeg}(\mesh)} \#\{\elm'\in\mesh \mid \elm'\ni z\}
\]
stands for the maximum number of elements in a star. The latter can be bounded 
in terms of the shape coefficient of $\mesh$, see the end 
of \S\ref{S:ldecoupling}, and the claim follows.
\end{proof}

Theorem \ref{T:elm-dec} covers various boundary conditions associated with an 
`$\Hil1$-setting'.  We illustrate this by discussing Dirichlet
and Neumann boundary conditions for Poisson's equation.

For Dirichlet boundary conditions, we follow the approach of
L.\ R.\ Scott and S.\ Zhang in \cite[\S5]{Scott.Zhang:90}.  Denote by
$\IpSZ:\Hil{1}(\Omega)\to\Splines^{\pdeg,0}(\mesh)$ the interpolation
operator therein and recall that the restriction $\IpSZ v_{|\partial\Omega}$
depends only on $v_{|\partial\Omega}$.  Given boundary values
$g\in\Hil{\frac12}(\partial\Omega)$, the weak solution of a Dirichlet
problem is from the trial space
\[
 X_g \defas \{ v \in \Hil1(\Omega) \mid v_{|\partial\Omega} = g \}
\]
and the finite element solution is sought in the space
\[
 \Splines_g(\mesh)
 \defas
 \{ V\in\Splines^{\pdeg,-1}(\mesh) \mid
    V\in C^0(\overline\Omega),
    V_{|\partial\Omega} = \IpSZ g \}
\]
which is not necessarily a subspace of $X_g$.  In view of \cite[Lemma 
2.1]{Sacchi.Veeser:06}, \eqref{H1seminorm} is a definite error notion on 
$X_g-\Splines_g(\mesh)$.  Since however $\IpSZ v_{|\partial\Omega} = 
v_{|\partial\Omega}$ for all $v\in\Splines^{\pdeg,0}(\mesh)$, there holds 
$\Splines_g(\mesh)\subset X_g$, i.e.\ the ensuing finite element method is 
conforming, whenever possible.  It is not difficult to show that the finite 
element solution is a near best approximation from $\Splines_g(\mesh)$.

\begin{cor}[Dirichlet boundary values]
\label{C:Dirichlet}
For any $v\in X_g$, there holds
\[
  E\big( v,\Splines_g(\mesh) \big)
 \leq
 C  \left[
  \sum_{\elm\in\mesh} e(v,\elm)^2
 \right]^{\frac12},
\]
where $C$ is the constant from Theorem \ref{T:LDcplng}.  Consequently, the 
decoupling coefficient of $\Splines_g(\mesh)$ is bounded in terms of $d$, 
$\pdeg$ and $\sigma_\mesh$.
\end{cor}

\begin{proof}
As for the case corresponding to the introduction, there holds
\[
 \Cnstr = \LgrNodes{\pdeg}(\mesh) \cap \Skeleton
\quad\text{with}\quad
 \Skeleton
 =
 \bigcup_{\Face\in\Faces(\mesh)} \Face
\]
and we require that $\Face_z\subset\partial\Omega$ whenever 
$z\in\Cnstr\cap\partial\Omega$.  Moreover, we require that the choices of 
$\Face_z$ for all $z\in\LgrNodes{\pdeg}(\mesh)\cap\partial\Omega$ in definitions 
of $\IpSZ$ and $\Ip$ coincide.  Consequently,
\[
 \Ip v_{|\partial\Omega}
  =
 \IpSZ v_{|\partial\Omega}
\]
and \eqref{Ip;BdryVals} holds with $X=X_g$ and 
$\Splines(\mesh)=\Splines_g(\mesh)$. Theorem \ref{T:elm-dec} thus yields the
claim.
\end{proof}
For homogeneous Dirichlet boundary values $g=0$, Corollary \ref{C:Dirichlet}
implies the non-obivous part of \eqref{equivalence}. A further immediate
consequence is that `near best in $\Splines_g(\mesh)$' entails `near best in 
$\Splines^{\pdeg,0}(\mesh)$'.  In particular, the aforementioned finite element 
solution is thus near best in $\Splines^{\pdeg,0}(\mesh)$.

For Neumann boundary conditions, the trial space is
\[
 \Tilde X \defas \Hil1(\Omega)/\R,
\]
which can be approximated by
\[
 \Tilde\Splines(\mesh)
 \defas
 \Splines^{\pdeg,0}(\mesh)/\R.
\]
Again, \eqref{H1seminorm} is a definite error notion and the finite element
solution of this space is near best.
 
\begin{cor}[Neumann boundary values]
\label{C:Neumann}
For any $v\in \Tilde X$, there holds
\[
  E\big( v, \Tilde\Splines(\mesh) \big)
 \leq
 C  \left[
  \sum_{\elm\in\mesh} e(v,\elm)^2
 \right]^{\frac12},
\]
where $C$ is the constant from Theorem \ref{T:LDcplng}.  Consequently,
the decoupling coefficient of $\Tilde\Splines(\mesh)$ is bounded
in terms of $d$, $\pdeg$ and $\sigma_\mesh$.
\end{cor}

\begin{proof}
 Identifying
 $\Hil1(\Omega)/\R$ and $\{v\in\Hil1(\Omega) \mid \int_\Omega v = 0 \}$,
 the implication \eqref{Ip;BdryVals} holds with $X=\Tilde X$ and Theorem
 \eqref{T:elm-dec} again yields the claim.
\end{proof}
Similarly, one can consider Robin boundary conditions or mixed ones and obtain
corresponding statements.
 
\begin{rem}[Construction of interpolation operator]
\label{R:Ip}
The definition of the nodal values $\Ip_zv$ for constrained nodes
$z\in\LgrNodes{\pdeg}(\mesh)\cap\Cnstr$ is the critical part.  For constrained
nodes on the boundary, we follow the approach of \cite{Scott.Zhang:90}.
The role of \eqref{nbnval} in proving Theorem \ref{T:elm-dec} reveals that this
is a near best choice with respect to the involved local errors and can be
adopted also for the other constrained nodes.  Inequality \eqref{nbnval} shows
also that the particular admissible choice of $\Face_z$ does not matter.
Moreover, its proof reveals that, for interior constrained nodes, also
$P_\elm(z)$ where the element $\elm$ contains $z$, or some average of these 
values may be used.  Interpolation operators of this type may be viewed as 
a composition of taking a best approximation in $S^{\ell,-1}(\mesh)$ and 
a so-called enriching operator.   The latter have been used to connect 
non-conforming finite element methods to conforming ones in various contexts; 
see, e.g., S.\ Brenner \cite{Brenner:96, Brenner:99}, O.\ S.\ Karakashian and 
F.\ Pascal \cite{Karakashian.Pascal:03}, T.\ Gudi \cite{Gudi:10} and A.\ Bonito 
and R.\ H.\ Nochetto \cite{Bonito.Nochetto:10}.
\end{rem}

\subsection{Local gradient conformity}
\label{S:grd-conformity}
%
The local best errors $e(v,\elm)$, $\elm\in\mesh$, are related to approximation 
problems of the following type. Approximate a vector function, which is the 
gradient of a scalar function, with gradients of polynomial functions.  The 
components of the approximants are coupled whenever $\pdeg\geq2$: indeed, the 
value of higher order partial derivatives then does not depend on their order of 
application.  Hence the following question arises: Does this coupling lead to 
some downgrading with respect to approximants, the components of which are 
independent polynomial functions?

A similar question arises for $E(v,\mesh)$.  However, in view of Theorem 
\ref{T:elm-dec}, it suffices for both questions to compare the local best 
errors $e(v,\elm)$, $\elm\in\mesh$, with the following ones:
\begin{align*}
 \ve(v,\elm)
 &\defas
 \inf_{Q\in\Poly{\pdeg-1}{\elm}^d} \verror{\grd v-Q}{\elm}
 =
 \left( \sum_{i=1}^d
  \inf_{R\in\Poly{\pdeg-1}{\elm}} \verror{\partial_i v - R}{\elm}^2
 \right)^{\frac12},
\end{align*}
with
\[
 \verror{f}{\Omega}
 \defas
 \| |f| \|_{0,2;\Omega}
 =
 \left(
  \int_\Omega |f|^2
 \right)^{\frac12}
\]
for $f\in\Leb2(\Omega)^d$.  Also here the inequality
\[
 \ve(v,\elm) \leq e(v,\elm)
\]
is straight-forward, while the opposite one is more involved and its proof 
relies on the construction of a suitable (quasi-)interpolant.  We shall use the 
averaged Taylor polynomial of \cite{Dupont.Scott:80} by T.~Dupont and 
L.~R.~Scott, which is a variant of the one of S.\ L. Sobolev.  In order to 
avoid a dependence on the element shape, we follow an idea of S.~Dekel and 
D.~Leviatan in \cite{Dekel.Leviatan:04} and average in a reference 
configuration. 
\begin{thm}[Decoupling of partial derivatives]
\label{T:pder-dec}
There is a constant $C$ depending only on $d$, $\pdeg\in\N$ such that, for any 
element $\elm\in\mesh$ and any function $v\in\Hil1(\elm)$, there holds
\[
  e(v,\elm)
  \leq
  C \ve(v,\elm).
\]
\end{thm}

\begin{proof}
Given $\pdeg\in\N_0$, denote by $I_\pdeg$ the operator corresponding to the 
averaged Taylor polynomial of order $\pdeg+1$ (and degree $\leq\pdeg)$ over the 
largest inscribed ball in $\Simplexd$.  Remarkably, the averaged Taylor 
polynomial commutes with differentation in that $\partial_i (I_\pdeg 
w)=I_{\pdeg-1} (\partial_i w)$ for all $w\in\Hil1(\Simplexd)$ and 
$i\in\{1,\dots,d\}$; see, e.g., \cite[(4.1.17)]{Brenner.Scott:08}.  Corollary 
3.4 in \cite{Dekel.Leviatan:04} generalizes this to
\begin{equation}\label{Diff-avgTaylor}
 \partial_i \Big( \big[ I_\pdeg (v\circ A) \big] \circ A^{-1} \Big)
 =
 \Big( I_{\pdeg-1} \big[ (\partial_i v) \circ A \big] \Big) \circ A^{-1},
\end{equation}
where $A:\Simplexd\to\elm$ is an affine bijection, $v\in\Hil1(\elm)$ and 
$\pdeg\geq1$.  Moreover, $I_\pdeg$ is a $\Leb2$-stable projection on 
$\Poly{\pdeg}{\Simplexd}$: for any $w\in\Leb2(\Simplexd)$, there hold
\begin{gather*}
 w\in\Poly{\pdeg}{\Simplexd} \implies I_\pdeg w = w,
\\
 \verror{I_\pdeg w}{\Simplexd} \leq C_{d,\pdeg} \verror{w}{\Simplexd}.
\end{gather*}
see, e.g., \cite[(4.1.15), (4.2.8)]{Brenner.Scott:08}.  Lemma~5 of J.~Xu and 
L.~Zikatanov \cite{Xu.Zikatanov:03} therefore implies that $I_\pdeg$ is near 
best with the constant $C_{d,\pdeg}$:
\begin{equation}\label{Ipdeg:nbest}
 \verror{w-I_\pdeg w}{\Simplexd}
 \leq
 C_{d,\pdeg} \inf_{S\in\Poly{\pdeg}{\Simplexd}} \verror{w-S}{\Simplexd}.
\end{equation}

Motivated by \eqref{Diff-avgTaylor}, we choose $P = \big[ I_\pdeg (v\circ A) 
\big] \circ A^{-1} \in\Poly{\pdeg}{\elm}$ and, also using the transformation 
rule and \eqref{Ipdeg:nbest} with $\pdeg-1$ in place of $\pdeg$, 
we obtain
\begin{align*}
 e(v,\elm)^2
 &\leq
 \verror{\grd (v - P)}{\elm}^2
 =
 \sum_{i=1}^d \verror{\partial_i v - \partial_i P}{\elm}^2
\\
 &=
 \sum_{i=1}^d 
  \verror{\partial_i v - \Big(
    I_{\pdeg-1} \big[ (\partial_i v) \circ A \big] 
    \Big) \circ A^{-1}}{\elm}^2.
\\
 &=
 \frac{\vol{\elm}}{\vol{\Simplexd}} \sum_{i=1}^d 
  \verror{ (\partial_i v)\circ A - I_{\pdeg-1} \big[ (\partial_i v) \circ A 
   \big] }{\Simplexd}^2
\\
 &\leq
 C_{d,\pdeg-1}^2 \frac{\vol{\elm}}{\vol{\Simplexd}} \sum_{i=1}^d
  \inf_{S\in\Poly{\pdeg-1}{\Simplexd}}
   \verror{\partial_i v\circ A - S}{\Simplexd}^2
\\
 &\leq
 C_{d,\pdeg-1}^2  \sum_{i=1}^d
  \inf_{R\in\Poly{\pdeg-1}{\elm}}
   \verror{\partial_i v - R}{\elm}^2
 =
 C_{d,\pdeg-1}^2 \ve(v,\elm)^2
\end{align*}
Consequently, the claimed inequality holds with $C=C_{d,\pdeg-1}$.
\end{proof}

The combination of Theorems \ref{T:elm-dec} and \ref{T:pder-dec} yields the 
following statement.

\begin{cor}[Decoupling of elements and partial derivatives]
\label{C:elm-pd-dec}
There is a constant $C$ such that for any $v\in X$ there holds
\[
 E\big( v,\Splines(\mesh) \big)
 \leq
 C \left[
  \sum_{\elm\in\mesh} \ve(v,\elm)^2
 \right]^{\frac12},
\]
with $C\leq C_{d,\pdeg-1} \DcConst(\mesh)$, which can be bounded in terms of 
$d$, $\pdeg$ and $\sigma_\mesh$.
\end{cor}

It is worth noticing that the right-hand side in Corollary \ref{C:elm-pd-dec} 
involves best errors of approximation problems that may be considered the most 
simplest ones involving polynomials and measuring the error in some 
$\Leb2$-sense.

\section{Two applications}
\label{S:appl}
%
%
The main novelty of the preceding section lies in the type of statements that 
are proven.  The goal of this section is to advocate its usefulness by 
showing that it allows for simplifications or improvements in theory and 
algorithms.  Doing so, we focus on applications of the decoupling of elements 
and, in order to cover various boundary conditions, we adopt the setting of 
\S\ref{S:best-errors}.

\subsection{Convergence and error estimates}
\label{S:Conv-ErrEst}
%
We start by reviewing some approximation results that play an important role in 
the a~priori error analysis of finite element methods.

A minimum requirement for a numerical method for a boundary value problem is 
that the approximate solution converges to the exact one as the meshsize tends 
to $0$.  In case of a finite element method, a necessary (and also sufficient, 
if for example the C\'ea Lemma holds) condition for this is that the best error 
of the corresponding finite element space tends to $0$.  It is instructive to 
prove this well-known fact for continuous piecewise polynomial functions with 
the help of Theorem \ref{T:elm-dec}.

\begin{thm}[Convergence]
There holds
 \[
  E(v,\mesh) \to 0
  \quad\text{as}\quad
  h \defas \max_{\elm\in\mesh} h_\elm \to 0
 \]
within a shape-regular family of meshes with $(d-1)$-face connected stars.
 \end{thm}

\begin{proof}
Given a vector function $f\in\Leb2(\Omega)^d$, let $\overline f\in 
\Leb\infty(\Omega)^d$ be the piecewise constant function given by
\[
 \forall\elm\in\mesh
 \quad
 \overline f_{|\elm} = \frac1{\vol{\elm}} \int_\elm f.
\]
Theorem \ref{T:elm-dec} and $\pdeg\geq1$ imply
\begin{equation}\label{mv-bd}\begin{aligned}
 E(v,\mesh) 
 &\leq
 C \left( \sum_{\elm\in\mesh} e(v,\elm)^2 \right)^{\frac12}
 \leq
 C \left(
  \sum_{\elm\in\mesh} \inf_{P\in\Poly1\elm} \error{v-P}{\elm}^2
 \right)^{\frac12}
\\
 &=
 C \left(
  \sum_{\elm\in\mesh} \inf_{c\in\R} \verror{\grd v - c}{\elm}^2
 \right)^{\frac12}
 =
 C \verror{\grd v - \overline{\grd v}}{\Omega},
\end{aligned}\end{equation}
which allows to conclude with a standard argument.  Let $\eps>0$ be arbitrary. 
Since $C^0(\overline\Omega)$ is dense in $\Leb2(\Omega)$, there exists $g\in 
C^0(\overline\Omega)^d$ such that $\verror{\grd v - g}{\Omega}\leq\eps/3$. 
Thanks to $\verror{\overline f}{\Omega}\leq\verror{f}{\Omega}$, we derive
\begin{multline*}
 \verror{\grd v - \overline{\grd v}}{\Omega}
 \leq
 \verror{\grd v - g}{\Omega}
  + \verror{g - \overline g}{\Omega}
  + \verror{\overline{g-\grd v}}{\Omega}
 \leq
 (2\eps)/3 + \verror{g - \overline g}{\Omega}
\end{multline*}
and the last term can be made smaller $\eps/3$ for sufficiently small $h$, 
because $g$ is uniform continuous in view of the compactness of 
$\overline\Omega$.
\end{proof}

Notice that the first equality in \eqref{mv-bd} corresponds to a special case 
of Theorem \ref{T:pder-dec} and its combination with Theorem \ref{T:elm-dec} 
simplifies the following density argument in that it does not involve 
derivatives.

\medskip Usually, the validity of a finite element method is theoretically 
investigated by deriving a~priori error estimates, quantifying the convergence 
speed in terms of powers of $h$.  Accordingly, such estimates for the best 
error of the corresponding finite element space are of interest.  These are 
usually obtained by directly bounding the error of some interpolation operator 
in terms of higher order Sobolev seminorms.  Here we use Theorem 
\ref{T:elm-dec} and
\begin{equation*}
 \Sobseminorm{v}{s}{2}{\elm}
 \defas
 \begin{cases}
  \left(
   \sum_{|\alpha|=s} \Lebnorm{\partial^\alpha v}{2}{\elm}^2
  \right)^{\frac12}
  & \text{if }s\in\N
\\
  (1-\theta)\left(
    \sum_{|\alpha|=\lfloor s\rfloor} \displaystyle
     \int_\elm\int_\elm
     \frac{|\partial^\alpha v(x) - \partial^\alpha v(y)|^2}{|x-y|^{2\theta+d}}
    \right)^{\frac12}
  & \text{otherwise,}
 \end{cases} 
\end{equation*}
where $s>0$ indicates the smoothness, $\theta \defas s - \lfloor s \rfloor$ 
its fractional part and the sums are over all multi-indices of order $\lfloor s 
\rfloor$; the factor $(1-\theta)$ is motivated by J.\ Bourgain et al.\ 
\cite{Bourgain.Brezis.Mironescu:01}.

\begin{thm}[Error estimates]
\label{T:ErrEst}
Let $v\in X$ a target function, $\mesh$ be a mesh with $(d-1)$-face connected 
stars and $1\leq s\leq\pdeg+1$.  If $v_{|\elm}\in\Hil{s}(\elm)$ for all 
$\elm\in\mesh$, then
\[
 E(v,\mesh)
 \leq
 C \left(
   \sum_{\elm\in\mesh} h_\elm^{2(s-1)} \Sobseminorm{v}{s}{2}{\elm}^2
 \right)^{\frac12},
\]
where $C$ can be bounded in terms of $d$, $\pdeg$, $\sigma_\mesh$ and $s$. 
If $s\in\N$, then
\[
 C \leq \frac{s!}{(\lceil\frac{s}{d}\rceil!)^d} C_P^{s-1}\delta(\mesh)
\]
where $C_P$ is the optimal Poincar\'e constant for $d$-simplices in $\Rd$.
\end{thm}

\begin{proof}
Combine Theorem \ref{T:elm-dec} and the Bramble-Hilbert inequality
\begin{equation}\label{BrambleHilbert}
 e(v,\elm)
 \leq
 C h_\elm^{s-1} \Sobseminorm{v}{s}{2}{\elm}
\end{equation}
where $C$ depends on $d$, $\pdeg$, $s$ and the shape coefficient of $d$-simplex 
$\elm$.  The latter follows, e.g., from \cite[Theorems 3.2 and 
6.1]{Dupont.Scott:80} and a standard scaling argument.  The explicit constant 
for $s\in\N$ is ensured by choosing a polynomial that allows an iterative 
application of the Poincar\'e inequality; see R.~Verf\"urth 
\cite[\S3]{Verfuerth:99}.
\end{proof}

We compare the bound of Theorem \ref{T:ErrEst} and those available via 
Lagrange and Scott-Zhang interpolation \cite[\S4]{Scott.Zhang:90}; error bounds 
via Cl\'ement interpolation \cite{Clement:75} coincide with the latter ones for 
homogeneous Dirichlet boundary conditions.

Doing so, we fix a mesh and vary through functions -- a viewpoint differing 
from the usual one where a function is fixed and meshes vary.  Notice that in 
the former case to be considered the most convenient choice of $s$ is not 
necessarily the maximal one; indeed, the product $h_\elm^s 
\Sobseminorm{v}{s}{2}{\elm}$ may not be monotone decreasing in $s$ for certain 
functions. This is closely related to the following drawback of the bounds via 
Lagrange interpolation for $d\geq2$.  Consider a sequence of functions 
$(v_n)_n$ in $\Hil{\lfloor d/2 \rfloor +1}(\Omega)$ converging to a function not 
better than $\Hil{d/2}(\Omega)$.  Then any bound via Lagrange interpolation 
blows up, while those of Theorem \ref{T:ErrEst} and via Scott-Zhang 
interpolation remain bounded for suitable $s\in\clsint1{d/2}$.

On the other hand, bounds via Lagrange interpolation can be readily derived 
in terms of broken Sobolev seminorms on elements; see, e.g., 
\cite[(4.4.20)]{Brenner.Scott:08} and modify the last three steps in its 
proof.  While Theorem \ref{T:ErrEst} shares this property, the bounds via 
Scott-Zhang interpolation \cite{Scott.Zhang:90} involve regularity across 
element boundaries.  The broken Sobolev norms on elements have the following 
advantage: If $E(v,\mesh)=0$, then the bounds of Theorem \ref{T:ErrEst} and 
Lagrange interpolation also vanish, while those via Scott-Zhang interpolation 
\cite{Scott.Zhang:90} do not vanish for $s\in\opnint1{3/2}$ and are not 
applicable, i.e.\ are $\infty$, for $s\geq3/2$ whenever $v$ has non-constant 
gradient. Similarly as before, this has its counterparts for smooth functions.  
To illustrate this, let $(v_n)_n$ be a sequence of functions in 
$\Hil{\pdeg+1}(\Omega)$ and consider various conditions ensuring convergence to 
a function in $\Splines(\mesh)$ with non-constant gradient.  If  $E(v_n,\mesh) 
\to 0$, then also the right-hand side of Theorem \ref{T:elm-dec} tends to $0$.  
If additionally the restrictions $v_n{}_{|\elm}$ converge in $\Hil{s}(\elm)$ 
for some $1 < s \leq \pdeg+1$, then the corresponding bound of Theorem 
\ref{T:ErrEst} tends to $0$.  The same holds for bounds of Lagrange 
interpolation, but only $s\in\ocint{d/2}{\pdeg+1}$ are admissible.  The 
situation for bounds via Scott-Zhang interpolation is different: if $1<s<3/2$, 
then the corresponding bound does not tend to $0$ and if $3/2\leq s \leq 
\pdeg+1$, it even blows up.

Summarizing, one may say that the approach via Theorem \ref{T:elm-dec} 
combines advantages of the available bounds for both Lagrange interpolation and 
interpolation of non-smooth functions. In particular, for 
$1 < s \leq d/2$ with $d\geq3$, the error bounds in terms of piecewise 
regularity seem to be new.  Bounds in terms of broken regularity are useful 
also in the context of surface finite element methods (SFEM); see 
F.\ Camacho and A.\ Demlow \cite{Camacho.Demlow:P}.

\subsection{Adaptive tree approximation of gradients}
\label{S:tree-approx}
%
The upper bound of the global best error in Theorem \ref{T:ErrEst} locally 
combines meshsize and higher order derivatives.  This suggests that, for a 
given target function, certain meshes a more convenient than others.
%
I.\ Babu\v{s}ka and W.\ C.\ Rheinboldt \cite{Babuska.Rheinboldt:78a} formally 
derive the following criteria, also called equidistribution principle: a mesh 
minimizing the aforementioned upper bound subject to a fixed number of elements 
equidistributes the element contributions, e.g.,  $h_\elm 
\Sobseminorm{v}{2}{2}{\elm}$ does not depend on $K\in\elm$.  Obviously this 
requires in general graded meshes.

An algorithmically simple way of constructing graded meshes arises from a 
prescribed rule for subdividing elements, which induces a tree structure.  An 
example of this was already in 1967 studied by S.\ Birman and M.\ Solomyak 
\cite{Birman.Solomyak:67}.  For continuous piecewise polynomial functions over 
conforming simplicial meshes, one can use bisection with recursive completion; 
for an overview of this technique, see e.g.\ Nochetto et al.\ 
\cite[\S4]{Nochetto.Siebert.Veeser:09}.  Although recursive bisection limits 
mesh flexibility, in particular mesh grading, the discussion in R.\ DeVore 
\cite[\S6]{DeVore:98} and the results of P.\ Binev et al.\ 
\cite{Binev.Dahmen.DeVore.Petrushev:02} reveal that the regularity dictating 
the asymptotic balance of global best error and number of degree of freedoms is 
close to the best possible one.  In particular, the global best $H^1$-error of 
suitably graded two-dimensional meshes decays like $\#\mesh^{-\frac12}$ if 
$u\in\Sob{2}{p}(\Omega)$ with $1<p\leq 2$; notice that the latter is weaker 
than the requirement $u\in\Hil2(\Omega)$ corresponding to the decay rate with 
quasi-uniform meshes and that $p=1$ corresponds to a optimal Sobolev embedding.

The goal of this section is to derive and analyze an instance of the tree 
algorithm by P.~Binev and R.~DeVore \cite{Binev.DeVore:04} that constructs near 
best bisection meshes for the approximation of gradients with piecewise 
polynomial functions.  It may be used for coarsening in adaptive algorithms 
iterating the main steps
 \[
  \texttt{error reduction} \to \texttt{sparsity adjustment}.
 \]
Interestingly, this scheme can be applied also if a good a posteriori error 
estimator is not available.  It includes algorithms like in 
\cite[\S8]{Binev.Dahmen.DeVore:04} and algorithms that are based upon 
discretizing the steps of an infinite-dimensional solver.  Moreover, the 
following algorithm can be used to compute an approximation of the best error of 
bisection meshes with a given number of elements. Such approximations are of 
interest as a benchmark for corresponding adaptive finite element methods.

In order to introduce the algorithm and to state its main property, we need the 
following notation.  Let $\mesh_0$ be an initial mesh of $\Omega$ that is 
admissible for bisection with recursive completion; see, e.g., \cite[Assumption 
11.1 on p.~453]{Nochetto.Siebert.Veeser:09}.  Denote by 
$\refmnts'$ the set of all meshes that can be generated by bisections without 
completion from $\mesh_0$; these meshes are not necessarily conforming and 
each one correponds to a subtree in the master tree given by $\mesh_0$ and the 
bisection rule of a single $d$-simplex.  Moreover, denote by $\refmnts$ the 
subset of $\refmnts'$ of all meshes that are conforming.  If 
$\mesh'\in\refmnts'$ is a possibly non-conforming mesh, denote by 
$\texttt{complete}(\mesh')$ the smallest refinement of $\mesh'$ in $\refmnts$.  
Since $\mesh_0$ is admissible, Binev et al.\ 
\cite[Theorem~2.4]{Binev.Dahmen.DeVore:04} if 
$d=2$ and R.~Stevenson \cite[Theorem~6.1]{Stevenson:08} otherwise ensure the 
non-obvious relationship
\begin{equation}\label{complete}
 \#\texttt{complete}(\mesh') - \#\mesh_0
 \leq
 \Ccmpl (\#\mesh' - \#\mesh_0)
\end{equation}
with $\Ccmpl$ depending only on $\mesh_0$.  For any $N\geq\#\mesh_0$, we 
associate best errors with the two mesh families $\refmnts$ and $\refmnts'$.  
Namely, the best approximation error 
\begin{equation*}
 \sigma(v,N)
 \defas
 \min \big\{ E \big( v,\Splines(\mesh) \big)
  \mid \mesh\in\refmnts,\ \#\mesh\leq 
N \big\}
\end{equation*}
with continuous piecewise polynomial functions over conforming meshes with less 
than $N$ elements, which is greater than the corresponding best error 
\begin{equation}
\label{nlberr}
 \sigma'(v,N)
 \defas
 \min \{ E\big( v,\Splines^{\pdeg,-1}(\mesh') \big)
   \mid \mesh'\in\refmnts',\ \#\mesh'\leq N \}
\end{equation}
with possibly discontinuous piecewise polynomial functions over possibly 
non-con\-form\-ing meshes.  There also holds an inequality in the oppposite 
direction, which may be seen as a generalization of Theorem \ref{T:elm-dec}.

\begin{thm}[Non-conforming decoupling of elements]
\label{T:n-elm-dec}
Assume that the initial mesh $\mesh_0$ is conforming, admissible, and that 
all its stars are $(d-1)$-face-connected.  Then there exist constants $C_1$ 
and $C_2$ depending only on $\mesh_0$ such that, for any $v\in X$ and 
$N\geq N_0\defas\#\mesh_0$, there holds
\[
 \sigma(v,N)
 \leq
 C_1 \sigma'\left(
  v, \left\lfloor \frac{N-N_0}{C_2} \right\rfloor + N_0
  \right).
\]
\end{thm}

\begin{proof}
Set $N'\defas\lfloor (N-N_0)/\Ccmpl \rfloor + N_0$ and choose an optimal 
possibly non-conforming mesh $\mesh\in\refmnts'$ such that 
$E\big( v,\Splines^{\pdeg,-1}(\mesh') \big) = \sigma'(v,N')$.  Since 
$\mesh_0$ is admissible, \eqref{complete} yields $\#\mesh\leq N$ for
$\mesh\defas\texttt{complete}(\mesh')$.  Hence
\begin{align*}
 \sigma(v,N)
 &\leq
 E\big( v, \Splines(\mesh) \big)
 \leq
 \DcConst(\mesh) E\big( v, \Splines^{\pdeg,-1}(\mesh) \big)
 \leq
 \DcConst(\mesh) E\big( v, \Splines^{\pdeg,-1}(\mesh') \big)
\\
 &=
 \DcConst(\mesh) \sigma'(v,N')
\end{align*}
thanks to $\Splines^{\pdeg,-1}(\mesh)\subset\Splines^{\pdeg,-1}(\mesh')$.  
The shape coefficient for any mesh in $\refmnts$ is bounded in terms of the one 
of $\mesh_0$; see, e.g., \cite[Lemma 4.1]{Nochetto.Siebert.Veeser:09}.   
Moreover, any mesh in $\refmnts$ inherits from $\mesh_0$ that all its stars are 
$(d-1)$-face-connected.  Theorem \ref{T:elm-dec} therefore implies 
$\DcConst(\mesh)\leq\DcConst$, were $\DcConst$ depends only on $\mesh_0$.  The 
claimed inequality thus holds with $C_1=\DcConst$ and $C_2=\Ccmpl$.
\end{proof}

The number of competing meshes for the best errors grows exponentially with 
$N-\#\mesh_0$.  Nevertheless the following variant of adaptive tree 
approximation by P.~Binev and R.~DeVore \cite{Binev.DeVore:04} constructs near 
best meshes with $O(N)$ operations and computations of the local error 
functional
\[
 \eps(\elm)\defas e(v,\elm)^2,
\]
where $\elm$ is any $d$-simplex and $v\in\Hil1(\Omega)$ the target function. 
Given a threshold $t>0$, we proceed as follows:

\smallskip\noindent
\begin{itemize}
 \item[] $\mesh_t'\defas\emptyset$;
 \item[] \texttt{for all} $\elm\in\mesh_0$
 \begin{itemize}
  \item[] $\eta(\elm) \defas \eps(\elm)$;
  \item[] $\texttt{if } \eta(\elm)>t \texttt{ then grow}(\elm)$;
 \end{itemize}
 \item[] \texttt{end for}
 \item[] $\mesh_t \defas \texttt{complete}(\mesh_t')$;
\end{itemize}

\smallskip\noindent
where $\texttt{grow}(\elm)$ grows the subtree generating $\mesh_t'$ and 
collects its leafs by

\smallskip\noindent
\begin{itemize}
 \item[] $(\elm_1,\elm_2)=\texttt{bisect}(\elm)$;
 \item[] \texttt{for} $i=1,2$
 \begin{itemize}
  \item[] $\eta(\elm_i) \defas \big[
      \eps(\elm_i)^{-1} + \eta(\elm)^{-1} 
    \big]^{-1}$
  \item[] $\texttt{if } \eta(\elm_i)>t$ \texttt{ then}
  \begin{itemize}
   \item[] $\texttt{grow}(\elm_i)$;
  \end{itemize}
  \item[] \texttt{else}
  \begin{itemize}
   \item[] $\mesh_t' \defas \mesh_t' \cap \{\elm_i\}$; 
  \end{itemize}
 \end{itemize}
\end{itemize}

\smallskip\noindent
and $\texttt{bisect}(\elm)$ implements the bisection of a single simplex; 
see, e.g., \cite[\S4.1]{Nochetto.Siebert.Veeser:09}.

The core of this algorithm is the thresholding algorithm in 
\cite{Birman.Solomyak:67} with the following important difference: the 
local functional $\eta(\elm)$ depends not only on the local error functionals 
but also on their history within in the subdivison hierarchy.

There are noteworthy variants of this algorithm.  In particular, the threshold 
$t$ can be avoided by successively bisecting the elements maximizing the 
indicators $\eta(\elm)$ of the current mesh; see \cite{Binev.DeVore:04}.  In 
this case one may also ensure the conformity of the mesh at any intermediate 
step.  For these variants, the following theorem presents only non-essential 
changes.

\begin{thm}[Tree approximation]
\label{T:tree-approx}
Assume that the initial mesh $\mesh_0$ is conforming, admissible, and that 
all its stars are $(d-1)$-face-connected.  Then there exist constants $C_1$ 
and $C_2$ depending only on $\mesh_0$ such that, for any $v\in X$ and any 
threshold $t>0$, the output mesh $\mesh$ of the tree algorithm verifies
\[
 E\big( v,\Splines(\mesh) \big)
 \leq
 C_1 \sigma' \left( v, \left\lceil \frac{\#\mesh}{C_2} \right\rceil \right)
\]
whenever $\#\mesh\geq C_2(2\#\mesh_0+1)$.
\end{thm}

\begin{proof}
The local error functional $\eps(\elm) = e(v,\elm)^2 = 
\inf_{P\in\Poly{\pdeg}{\elm}} \error{v-P}{\elm}^2$ obviously depends only on 
the target function $v$ and the simplex $\elm$.  Moreover, it is subadditive: 
if $(\elm_1,\elm_2)=\texttt{bisect}(\elm)$, then $\eps(\elm_1) + \eps(\elm_2) 
\leq \error{v-P}{\elm_1}^2 + \error{v-P}{\elm_2}^2$ for any 
$P\in\Poly{\pdeg}{\elm}$ and thus
\[
 \eps(\elm_1) + \eps(\elm_2) \leq \eps(\elm).
\]
Hence Theorem 4 of P.~Binev's contribution in \cite{Oberwolfach4.3:07} 
applies to the above for-loop that constructs $\mesh_t'$.  Writing
\[
 N_0\defas\#\mesh_0,
\quad
 N'\defas\#\mesh_t',
\quad
 L'\defas N'-N_0
\]
and observing $E\big( v,\Splines^{\pdeg,-1}(\mesh) \big)^2 = 
\sum_{\elm\in\mesh} \eps(\elm)$ for any $\mesh'\in\refmnts'$, we therefore have
\[
 E\big( v,\Splines^{\pdeg,-1}(\mesh_t') \big)
 \leq
 \min_{0 \leq l \leq L'} \left( 1 +
   \frac{l + \min\{l,N_0\}}{L' + 1 - l}
   \right) \sigma'(v,N_0+l).
\]
Under the assumption $N'\geq2N_0$ this simplifies to
\[
 E\big( v,\Splines^{\pdeg,-1}(\mesh_t') \big)
 \leq
 \min_{2N_0 \leq n \leq N'} \frac{N'+1}{N'+1-n} \sigma'(v,n).
\]
Since $\mesh_0$ is admissible, \eqref{complete} ensures
\[
 N - N_0
 \leq
 \Ccmpl (N' - N_0)
\]
for $N\defas\#\mesh_t$.  In view of $N\geq N'\geq 2N_0$, we may use the simpler 
inequality $N\leq \Tilde C_2 N'$ with $\Tilde C_2=2\Ccmpl$.  Since $x\mapsto 
x/(x-c)$ is decreasing and $E$ monotone, one thus can derive
\[
 E\big( v,\Splines^{\pdeg,-1}(\mesh_t) \big)
 \leq
 \min_{2N_0 \leq n \leq N/\Tilde C_2}
   \frac{N}{N - \Tilde C_2 n} \sigma'(v,n)
\]
whenever $N\geq 2\Tilde C_2 N_0$.  Similarly as in the proof of Theorem 
\ref{T:n-elm-dec}, Theorem \ref{T:ErrEst} therefore implies
\[
 E\big( v,\Splines(\mesh_t) \big)
 \leq
 \delta \min_{2N_0 \leq n \leq N/\Tilde C_2}
   \frac{N}{N - \Tilde C_2 n} \sigma'(v,n)
\]
where $\delta$ depends only on $\mesh_0$.  Choosing $n\in\N$ such that
$N/(2\Tilde C_2) -1 < n \leq N/(2\Tilde C_2)$, the claimed inequality 
follows with $C_1=2\delta$ and $C_2=2\Tilde C_2$ as above.
\end{proof}

Theorem \ref{T:tree-approx} is of non-asymptotic nature; as can be seen from 
the proof, the condition $\#\mesh\geq C_2(2\#\mesh_0+1)$ arising from 
$N'\geq2N_0$ is of simplifying nature.  In particular, it does not suppose any 
regularity beyond $\Hil1(\Omega)$ of the target function.  

In \cite[\S7]{Binev.DeVore:04} a similar algorithm is proposed.  It relies on 
the local error functionals
\[
 \Tilde\eps(\elm)
 \defas
 \inf\big\{
  \error{v-V}{\omega_\elm} \mid V\in\Splines(\mesh_\elm)
  \big\}
\]
where $\mesh_\elm$ is the set of elements of the so-called minimal ring 
$R^-(\elm)$ around $\elm$ given by
\[
 R^-(\elm)
 \defas
 \bigcap_{\mesh\in\refmnts:\mesh\ni\elm} R(\elm,\mesh)
\quad\text{with}\quad
 R(\elm,\mesh)
 \defas
 \bigcup_{\elm'\in\mesh:\elm'\cap\elm\neq\emptyset} \elm'.
\]
In view of the minimality of the ring, these error functionals do not depend on 
the surrounding mesh.  They are not subadditive, but weakly subadditive with 
respect to repeated bisections and so another variant of the tree algorithm 
with a similar statement to Theorem \ref{T:tree-approx} still applies.  The 
local error functionals $e(\elm)$ are however simpler to implement and can be 
combined with single bisection.  Thus, the use of Theorem \ref{T:elm-dec} here 
permits an algorithmic simplification.

\subsection*{Acknowledgments}
The author wants to thank Francesco Mora for useful discussions regarding 
Theorem \ref{T:elm-dec}.

\end{document}